\documentclass[12pt]{article}
\usepackage[latin9]{inputenc}
\usepackage[a4paper]{geometry}
\usepackage{array}
\usepackage{float}
\usepackage{booktabs}
\usepackage{mathtools}
\usepackage{multirow}
\usepackage{amsmath}
\usepackage{amsthm}
\usepackage{amssymb}
\usepackage{graphicx}
\usepackage[authoryear]{natbib}
\usepackage[unicode=true,
 bookmarks=false,
 breaklinks=false,pdfborder={0 0 1},backref=section,colorlinks=false]
 {hyperref}

\makeatletter

\newcommand*\LyXThinSpace{\,\hspace{0pt}}
\providecommand{\tabularnewline}{\\}
\floatstyle{ruled}
\newfloat{algorithm}{tbp}{loa}
\providecommand{\algorithmname}{Algorithm}
\floatname{algorithm}{\protect\algorithmname}

\usepackage{amsthm}\usepackage{bm}
\usepackage{algorithm}
\usepackage{algorithmic}
\usepackage{cite}
\usepackage{authblk}
\usepackage{amsthm}
\usepackage[inline]{enumitem}
\usepackage{subcaption}
\usepackage{cleveref}
\newtheorem{theorem}{Theorem}

\title{Pontryagin-Guided Direct Policy Optimization Framework for Merton's Portfolio Problem}

\author[1]{Jeonggyu Huh}
\author[2]{Jaegi Jeon\thanks{Corresponding Author: jaegijeon@jnu.ac.kr}}
\affil[1]{\small Department of Mathematics, Sungkyunkwan University, Republic of Korea}
\affil[2]{\small Graduate School of Data Science, Chonnam National University, Republic of Korea}

\makeatother

\begin{document}
\title{Pontryagin-Guided Policy Optimization for Merton's
Portfolio Problem}
\maketitle
\begin{abstract}
We present a Pontryagin-Guided Direct Policy Optimization (PG-DPO)
framework for Merton's portfolio problem, unifying modern neural-network-based
policy parameterization with the adjoint viewpoint from Pontryagin's
maximum principle (PMP). Instead of approximating the value function
(as done in deep BSDE methods), we track a policy-fixed BSDE for the adjoint
processes, which allows each gradient update to align with continuous-time
PMP conditions. This setup yields locally optimal consumption and
investment policies that are closely tied to classical stochastic
control. We further incorporate an alignment penalty that nudges the
learned policy toward Pontryagin-derived solutions, enhancing both
convergence speed and training stability. Numerical experiments confirm
that PG-DPO effectively handles both consumption and investment,
achieving strong performance and interpretability without requiring
large offline datasets or model-free reinforcement learning. 
\end{abstract}


\section{Introduction}

Merton's portfolio optimization problem \citep{merton1971optimum}
is a cornerstone of mathematical finance, aiming to specify optimal
investment and consumption decisions in a continuous-time market.
Classical treatments \citep[e.g.,][]{karatzas1998methods,yong2012stochastic,pham2009continuous}
exploit the problem's analytical tractability to derive closed-form
solutions under specific assumptions, such as constant coefficients
or complete markets. In real-world settings, however, these idealized
conditions are often not met, prompting practitioners to rely on data-driven
methods that do not assume closed-form solutions. By parameterizing
the policy (both investment and consumption) with neural networks
and defining a suitable objective, one can use gradient-based
learning (e.g., stochastic gradient descent) to iteratively improve
the policy. Such deep-learning-based approaches have shown promise
for high-dimensional state spaces, path dependence, and market frictions
\citep{han2016deep,han2018solving,beck2019machine,buehler2019deep,becker2019deep,zhang2019deep,reppen2023deep_df,reppen2023deep_mf}.

Despite these advances, a purely empirical or black-box gradient
approach may overlook the continuous-time optimal control perspective.
Without classical stochastic control principles as guidance, there
is no clear guarantee that iterative updates will converge to a Pontryagin-aligned
policy, that is, one satisfying the adjoint-based optimality conditions
in continuous time. Recent studies \citep[e.g.,][]{reppen2023deep_df,reppen2023deep_mf}
highlight this tension: while purely data-driven methods can achieve
high performance in practice, they may not incorporate the deeper
theoretical underpinnings of optimal control. Other works, such as
\citet{dai2023learning}, adopt model-free reinforcement learning
(RL), which requires extensive exploration and often omits consumption
to focus on investment alone.

A different perspective is offered by the so-called deep BSDE (backward
stochastic differential equation) methodology, pioneered by Weinan
E and collaborators \citep[e.g.,][]{han2018solving,e2017deep}, which
focuses on approximating the value function (i.e., solving the Hamilton-Jacobi-Bellman
(HJB) equation) via a forward-backward stochastic differential equation
(FBSDE), then infers the policy once the value function is estimated.
By contrast, our approach is fundamentally adjoint-based: rather than
approximating the value function, we track adjoint (costate) processes
through a policy-fixed BSDE and use them to guide parameter updates
directly. In addition, we implement this scheme via backpropagation-through-time
(BPTT), unrolling the system dynamics in discrete steps and enabling
single-path simulation for each node. While both deep BSDE and our
Pontryagin-guided approach employ neural networks and stochastic differential
equations (SDEs), the difference in what is learned (value function
vs. adjoint) and how it is integrated (decoupled vs. direct BPTT)
can yield distinct interpretability and computational trade-offs.

However, to the best of our knowledge, no existing deep-BSDE or physics-informed neural networks (PINNs) 
approach has tackled the Merton problem with both consumption and investment
in a single framework. Most related works address either an investment-only
Merton setup or simpler PDEs (e.g., Black--Scholes) via PINNs \citep{han2018solving,hure2020deep}. This gap
further motivates our approach, which integrates Pontryagin's
principle with a fully neural-network-based solver that jointly handles
both consumption and investment.

Our framework adopts what we call a Pontryagin-Guided Direct Policy
Optimization (PG-DPO) viewpoint. We treat the Merton problem as a
continuous-time stochastic control system but solve it through direct
neural-network-based optimization. Specifically, we embed Pontryagin's
Maximum Principle (PMP) \citep{pontryagin1962the} into a discrete-time
training loop, interpreting each gradient step as an approximation
to the continuous-time adjoint. This preserves the convenience of
automatic differentiation and mini-batch simulation while offering
a principled control-theoretic justification of policy improvement.
We also incorporate both consumption and investment into the policy,
a combination seldom tackled simultaneously by purely data-driven
methods yet crucial in real-world portfolio planning.

Our numerical experiments and theoretical analyses demonstrate that
this direct policy optimization (DPO) \citep{reppen2023deep_df,reppen2023deep_mf}
approach converges to stationary policies that are strongly aligned
with Pontryagin's conditions. Moreover, viewing the adjoint through
the Pontryagin lens facilitates an alignment penalty (or regularization)
that can speed up convergence. This additional penalty encourages
the network's consumption and investment policies to stay near Pontryagin-derived
controls at each time instant, thereby improving training efficiency
and stability in practice. In many cases, the policy converges to
a local optimum that remains consistently Pontryagin-aligned and does
so without relying on large offline datasets or purely model-free
RL techniques.

Several aspects of this paper can be summarized as follows. First,
we propose a PG-DPO scheme that incorporates PMP in a discrete-time
neural network framework. Each gradient update is interpreted in terms
of the adjoint perspective, providing a more transparent
route to near-optimal solutions than opaque gradient-based updates
alone. Second, we add an adjoint-based alignment penalty, encouraging
the neural policy to stay close to locally Pontryagin-optimal controls
and thereby taking advantage of suboptimal adjoint estimates for faster
convergence. Although we only guarantee local optimality, this soft
constraint often yields notable improvements in training efficiency
and stability. Third, by directly tackling both consumption and investment
in the Merton problem, we address a more challenging scenario than
usual in data-driven approaches. Our method converges reliably, uses
rolling simulations (thus mitigating overfitting), and handles the
entire state-time domain with a single pipeline. Finally, the Pontryagin
viewpoint supports interpretability by making the adjoint processes
explicit, while still preserving the flexibility of modern deep-learning
tools. This synergy can naturally extend to more general setups involving
jumps, constraints, or multi-asset portfolios.

Section~\ref{sec:Merton} reviews the continuous-time
Merton formulation, emphasizing the interplay of consumption and investment.
Section~\ref{sec:PMP} revisits PMP and re-expresses
Merton's SDE via a forward-backward viewpoint, setting the stage
for our adjoint-based approach. In Section~\ref{sec:deep_BSDE_comparison},
we contrast our method with deep BSDE frameworks that approximate
the value function instead of the adjoint. Next, Section~\ref{sec:BPTT}
explains the discrete-time updates that approximate the continuous-time
Pontryagin flow, including an alignment penalty for faster convergence.
In Section~\ref{sec:convergence_analysis}, we use
a stochastic approximation viewpoint (Robbins-Monro theory) to show
that our method converges locally. Section~\ref{sec:num_test}
presents numerical experiments demonstrating convergence to robust
local optima aligned with Pontryagin's conditions. Finally, Section~\ref{sec:conclusion}
offers concluding remarks and discusses possible extensions that further
bridge deep learning and stochastic control.

\section{Continuous-Time Formulation of the Merton Portfolio Problem\label{sec:Merton}}

We focus on Merton's portfolio optimization problem,
in which an investor allocates wealth between a risky asset $S_{t}$
and a risk-free asset $B_{t}$. The risky asset follows a geometric
Brownian motion:
\[
dS_{t}=\mu S_{t}\,dt+\sigma S_{t}\,dW_{t},
\]
where $W_{t}$ is the standard Brownian motion, and the risk-free
asset evolves as 
\[
dB_{t}=rB_{t}\,dt.
\]
Let $X_{t}$ denote the investor's total wealth at
time $t$. The investor continuously consumes at rate $C_{t}$ and
invests proportion $\pi_{t}$ of $X_{t}$ in the risky asset, so that
$X_{t}$ satisfies
\begin{equation}
dX_{t}=\left(rX_{t}+\pi_{t}\left(\mu-r\right)X_{t}-C_{t}\right)\,dt+\pi_{t}\sigma X_{t}\,dW_{t},\quad X_{0}=x_{0}>0.\label{eq:SDE_state_X}
\end{equation}
We consider the objective function: 
\begin{equation}
J\left(\pi_{t},C_{t}\right)\coloneqq\mathbb{E}\left[\int_{0}^{T}e^{-\rho t}U\left(C_{t}\right)\,dt+\kappa e^{-\rho T}U\left(X_{T}\right)\right],\label{eq:objective}
\end{equation}
where 
\begin{equation}
U(x)=\frac{x^{1-\gamma}}{1-\gamma}\label{eq:CRRA_utility}
\end{equation}
is a CRRA (constant relative risk aversion) utility function, $\rho>0$
is a discount factor, and $\kappa>0$ is a parameter that modulates
the relative weight of terminal wealth (bequest). We parametrize the
investment and consumption via neural networks:
\[
\pi_{t}=\pi_{\theta}\left(t,X_{t}\right),\quad C_{t}=C_{\phi}\left(t,X_{t}\right),
\]
where $\theta$ and $\phi$ denote the network parameters to be learned.

In a continuous-time stochastic control framework, these controls
$\left\{ \pi_{t},C_{t}\right\} $ must be admissible. Concretely,
this means each $\pi_{t}$ and $C_{t}$ must be adapted and measurable
with respect to the filtration generated by $\left\{ W_{t}\right\} $,
thereby depending only on information available up to time $t$. They
also require boundedness or controlled growth (e.g. $\vert\pi_{t}\vert\leq M_{\pi}$
, $C_{t}\geq0$) to prevent explosive behavior in the state dynamics
and ensure the integrals in the objective function \eqref{eq:objective}
remain finite. Moreover, the choices of $\left\{ \pi_{t},C_{t}\right\} $
should preserve the nonnegativity of the wealth process $X_{t}$ almost
surely. Under these conditions, the SDE governing $X_{t}$ admits
a unique strong solution, and the objective $J\left(\pi_{t},C_{t}\right)$
is well-defined. In practice, one can fulfill these feasibility constraints
by selecting suitable activation functions or output layers in the
neural networks $\pi_{\theta}$ and $C_{\phi}$, ensuring that every sampled
trajectory respects the required conditions automatically.

A key property of the Merton problem is the strict concavity of the
CRRA utility \eqref{eq:CRRA_utility}. This strict concavity guarantees
the uniqueness of the global maximum and motivates classical results
\citep{merton1971optimum,karatzas1998methods,yong2012stochastic}
indicating that any stationary point of $J\left(\pi_{t},C_{t}\right)$
is, in fact, the unique global maximizer. This uniqueness is central
to the theoretical analysis we develop. 

In the classical Merton problem with CRRA utility and no bequest ($\kappa=0$),
closed-form solutions are well-known. For example, in the infinite-horizon
setting with discount rate $\rho>0$, the optimal investment and consumption
rules reduce to constant values:
\begin{equation}
\pi_{t}^{*}=\frac{\mu-r}{\gamma\sigma^{2}},\quad C_{t}^{*}=\nu X_{t},\quad\text{where }\nu=\frac{\rho-(1-\gamma)\left(\frac{(\mu-r)^{2}}{2\sigma^{2}\gamma}+r\right)}{\gamma}.\label{eq:optimal_inf_horizon}
\end{equation}
This form shows that a constant fraction of wealth is invested in
the risky asset, and consumption is proportional to current wealth
at a constant rate $\nu$.

In more general settings, such as a finite time horizon $T$ or
when the bequest term $\kappa>0$ is included, the closed-form solutions
become slightly more involved but remain explicit. For instance, under
a finite horizon $T$ and $\kappa>0$, the optimal consumption satisfies
\[
C_{t}^{*}=\nu\left(1+\left(\nu\epsilon-1\right)e^{-\nu\left(T-t\right)}\right)^{-1}X_{t},\quad\text{where }\epsilon=\kappa^{1/\gamma}
\]
and $\nu$ is the same constant defined in \eqref{eq:optimal_inf_horizon}.
As $T\to\infty$, this expression converges to the infinite-horizon
solution in \eqref{eq:optimal_inf_horizon}. These analytical formulas
serve as exact solutions that confirm our neural network-based policy
indeed converges to the theoretically optimal strategy. We shall use
these formulas in Section~\ref{sec:num_test} as a
numerical reference.

\section{Pontryagin's Maximum Principle and Its Application to Merton's
Problem\label{sec:PMP}}

\subsection{Pontryagin's Maximum Principle\label{sec:PMP_overview}}

PMP provides necessary conditions
for optimality in a broad class of continuous-time stochastic control
problems (see, e.g., \citealp{pontryagin1962the,pardouxpeng1990adapted,fleming2006controlled,yong2012stochastic,pham2009continuous}).
To illustrate this principle in a general setting, consider a system
whose state $X_{t}$ evolves according to the SDE
\begin{equation}
dX_{t}=b\left(t,X_{t},u_{t}\right)\,dt+\sigma\left(t,X_{t},u_{t}\right)\,dW_{t},\quad X_{0}=x_{0},\label{eq:SDE_PMP}
\end{equation}
where $u_{t}$ is an admissible control. The goal is to maximize an
objective functional of the form
\[
J\left(u\right)=\mathbb{E}\left[\int_{0}^{T}g\left(t,X_{t},u_{t}\right)\,dt+G\left(X_{T}\right)\right].
\]
If there exists an optimal control $u_{t}^{*}$ that attains the supremum
of $J\left(u\right)$, we denote by $J^{*}$ the corresponding optimal
value function. PMP is invoked by
constructing a Hamiltonian
\[
\mathcal{H}\left(t,X_{t},u_{t},\lambda_{t},Z_{t}\right)=g\left(t,X_{t},u_{t}\right)+\lambda_{t}b\left(t,X_{t},u_{t}\right)+Z_{t}\sigma\left(t,X_{t},u_{t}\right),
\]
along with adjoint (costate) processes $\lambda_{t}$ and $Z_{t}$.
Under the optimal control $u_{t}^{*}$, the state process $X_{t}^{*}$
is simply the unique solution of the original SDE \eqref{eq:SDE_PMP},
where $\left(X_{t},u_{t}\right)$ is replaced by $\left(X_{t}^{*},u_{t}^{*}\right).$
Formally, 
\begin{equation}
dX_{t}^{*}=b\left(t,X_{t}^{*},u_{t}^{*}\right)\,dt+\sigma\left(t,X_{t}^{*},u_{t}^{*}\right)\,dW_{t},\quad X_{0}^{*}=x_{0}.\label{eq:FSDE_PMP}
\end{equation}
In this setting, the adjoint $\lambda_{t}^{*}$ can be viewed as $\lambda_{t}^{*}=\frac{\partial J^{*}}{\partial X_{t}}$,
capturing how marginal changes in $X_{t}$ affect the optimal value.
In turn, $\left(\lambda_{t}^{*},Z_{t}^{*}\right)$ satisfies a BSDE
\begin{equation}
d\lambda_{t}^{*}=-\frac{\partial\mathcal{H}}{\partial X}\left(t,X_{t}^{*},u_{t}^{*},\lambda_{t}^{*},Z_{t}^{*}\right)\,dt+Z_{t}^{*}\,dW_{t},\quad\lambda_{T}^{*}=\frac{\partial G}{\partial X}\left(X_{T}^{*}\right).\label{eq:BSDE_PMP}
\end{equation}
Solving the BSDE \eqref{eq:BSDE_PMP} together with the forward SDE
\eqref{eq:FSDE_PMP} for $X_{t}^{*}$ forms a coupled FBSDE system. From this system, one obtains not only the optimal state
trajectory $X_{t}^{*}$ and the adjoint processes $(\lambda_{t}^{*},Z_{t}^{*})$,
but also characterizes the optimal control $u_{t}^{*}$ by maximizing
the Hamiltonian with respect to $u$ at each time $t$. Through this
local maximization step, PMP effectively transforms the global control
problem into a condition that must hold pointwise in time.

As a result, it offers a systematic way to compute (or approximate)
the optimal control, unifying value sensitivities $\lambda_{t}^{*}$,
noise sensitivities $Z_{t}^{*}$, and the evolving state $X_{t}^{*}$
in one mathematical framework \citep{fleming2006controlled,yong2012stochastic,pham2009continuous}.
In the sections that follow, we will see how this general procedure
applies specifically to the Merton portfolio problem, illustrating
how adjoint processes guide both investment and consumption decisions
in a continuous-time setting.

\subsection{Pontryagin Adjoint Processes and Parameter Gradients in the Merton
Problem\label{sec:adjoint_and_gradients}}

In this subsection, we combine two key components for handling both
the optimal Merton problem and suboptimal parameterized policies.
First, we restate the Pontryagin conditions for the Merton setup,
which yield the optimal relationships among $(\lambda_{t}^{*},Z_{t}^{*},\pi_{t}^{*},C_{t}^{*})$.
Second, we explain how to generalize these adjoint processes to a
policy-fixed BSDE in the suboptimal setting and derive corresponding
parameter gradients.

\subsubsection{Pontryagin Formulation for Merton's Problem}

Recall that in the Merton problem, we aim to maximize 
\[
J\left(\pi_{t},C_{t}\right)=\mathbb{E}\left[\int_{0}^{T}e^{-\rho t}U\left(C_{t}\right)\,dt+\kappa e^{-\rho T}U\left(X_{T}\right)\right],
\]
subject to 
\[
dX_{t}=(rX_{t}+\pi_{t}(\mu-r)X_{t}-C_{t})\,dt+\pi_{t}\sigma X_{t}\,dW_{t},\quad X_{0}=x_{0}>0
\]
where $U(\cdot)$ is a CRRA utility function, $\rho>0$ is a discount
rate, and $\kappa>0$ represents the bequest weight. To apply PMP
in this setting, we define the Hamiltonian 
\[
\mathcal{H}\left(t,X_{t},\pi_{t},C_{t},\lambda_{t},Z_{t}\right)=e^{-\rho t}U\left(C_{t}\right)+\lambda_{t}\left[rX_{t}+(\mu-r)\pi_{t}X_{t}-C_{t}\right]+Z_{t}\,\pi_{t}\sigma X_{t}.
\]
From PMP, the optimal controls $\left(\pi_{t}^{*},C_{t}^{*}\right)$
maximize $\mathcal{H}$ pointwise in $\left(\pi_{t},C_{t}\right)$.
Differentiating with respect to $C_{t}$ and $\pi_{t}$ and setting
these partial derivatives to zero yields:
\begin{enumerate}[label=(\alph*)]
\item \textbf{Consumption:} 
\[
\frac{\partial\mathcal{H}}{\partial C_{t}}=e^{-\rho t}U'\left(C_{t}^{*}\right)-\lambda_{t}=0\quad\Longrightarrow\quad U'(C_{t}^{*})=e^{\rho t}\lambda_{t}^{*}.
\]
Since $U'(x)=x^{-\gamma}$, it follows that 
\begin{equation}
C_{t}^{*}=\Bigl(e^{\rho t}\lambda_{t}^{*}\Bigr)^{-\frac{1}{\gamma}}.\label{eq:C_optimal}
\end{equation}
\item \textbf{Investment:} 
\[
\frac{\partial\mathcal{H}}{\partial\pi_{t}}=\left(\mu-r\right)\lambda_{t}^{*}X_{t}^{*}+Z_{t}^{*}\sigma X_{t}^{*}=0\quad\Longrightarrow\quad(\mu-r)\lambda_{t}^{*}+\sigma Z_{t}^{*}=0.
\]
Thus, the optimal adjoint processes satisfy $Z_{t}^{*}=-\frac{(\mu-r)}{\sigma}\,\lambda_{t}^{*}$. 
\end{enumerate}
Hence, the pair $\left(\lambda_{t}^{*},Z_{t}^{*}\right)$ links directly
to the optimal controls $\left(\pi_{t}^{*},C_{t}^{*}\right)$. In
practice, $\lambda_{t}^{*}$ can be interpreted as the sensitivity
of the optimal cost-to-go with respect to the state $X_{t}$, while
$Z_{t}^{*}$ captures the sensitivity to noise.

\paragraph{Remark (Expressing $\pi_{t}^{*}$ in terms of $\lambda_{t}^{*}$
and $\,Z_{t}^{*}$ under regularity).}

In many one-dimensional settings, including the Merton problem, an
additional smoothness assumption on the FBSDE
system leads to the relation
\[
Z_{t}^{*}=\left(\partial_{x}\lambda_{t}^{*}\right)\sigma\pi_{t}^{*}X_{t}^{*},
\]
where $\partial_{x}\lambda_{t}^{*}$ denotes the spatial derivative
of $\lambda_{t}^{*}$ with respect to the wealth variable $x$. Dividing
both sides by $\lambda_{t}^{*}$ and using $\left(\mu-r\right)\lambda_{t}^{*}+\sigma Z_{t}^{*}=0,$we
obtain 
\[
-\frac{\mu-r}{\sigma}=\frac{Z_{t}^{*}}{\lambda_{t}^{*}}=\frac{\partial_{x}\lambda_{t}^{*}}{\lambda_{t}^{*}}\sigma\pi_{t}^{*}X_{t}^{*}.
\]
Hence, under suitable conditions (e.g., if $\partial_{x}\lambda_{t}^{*}/\lambda_{t}^{*}$
remains bounded and nondegenerate), one can solve for $\pi_{t}^{*}$
explicitly: 
\begin{equation}
\pi_{t}^{*}=\frac{\mu-r}{\sigma^{2}}\times\frac{-\lambda_{t}^{*}}{\partial_{x}\lambda_{t}^{*}}\frac{1}{X_{t}^{*}}.\label{eq:pi_optimal}
\end{equation}
In the classical Merton problem with CRRA utility \eqref{eq:CRRA_utility},
solving the entire forward--backward system reveals that $\partial_{x}\lambda_{t}^{*}$
is proportional to $\lambda_{t}^{*}$, Equivalently, the ratio
\[
\frac{-\partial_{x}\lambda_{t}^{*}}{\lambda_{t}^{*}}X_{t}^{*}=\gamma
\]
becomes a constant in both time and wealth. This reflects the well-known
fact that CRRA utility fixes the investor's relative
risk aversion to $\gamma$, thereby forcing $\pi_{t}^{*}$ to be a
constant fraction of wealth:
\[
\pi_{t}^{*}=\frac{\mu-r}{\gamma\sigma^{2}}.
\]
But more generally, even if $\pi_{t}^{*}$ depends on time or wealth
in a nontrivial way, the above relationship shows how one can characterize
it through the ratio $\tfrac{Z_{t}^{*}}{\lambda_{t}^{*}}$ (and partial
derivatives of $\lambda_{t}^{*}$), once the full forward--backward
system is solved.

\subsubsection{Policy-Fixed Adjoint Processes for Suboptimal Policies}

In classical PMP theory, the adjoint processes $\lambda_{t}^{*}$
and $Z_{t}^{*}$ naturally emerge from the FBSDE system associated
with the optimal policy that maximizes the Hamiltonian. However, one
can extend these adjoint processes to suboptimal policies as
well. This extended notion is often called the policy-fixed adjoint,
because the policy is treated as given (hence, fixed) rather
than being chosen to maximize $\mathcal{H}$ in each infinitesimal
step.

Suppose we have a parameterized, potentially suboptimal policy $\left(\pi_{t},C_{t}\right)$
that does not necessarily satisfy the first-order optimality conditions.
We can still define a Hamiltonian $\mathcal{H}$ with the terminal
cost $G\left(X_{T}\right)$, which might be, for instance, $\kappa e^{-\rho T}U\left(X_{T}\right)$
as in the Merton problem.

Although $\left(\pi_{t},C_{t}\right)$ is not chosen to maximize $\mathcal{H}$,
we may formally differentiate $\mathcal{H}$ with respect to the state
$X$ and postulate a BSDE for $\lambda_{t}$ and $Z_{t}$:

\begin{equation}
d\lambda_{t}=-\frac{\partial\mathcal{H}}{\partial X}\left(t,X_{t},\pi_{t},C_{t},\lambda_{t},Z_{t}\right)dt+Z_{t}dW_{t},\quad\lambda_{T}=\frac{\partial G}{\partial X}(X_{T}).\label{eq:BSDE_suboptimal}
\end{equation}

Since the policy $\left(\pi_{t},C_{t}\right)$ does not maximize $\mathcal{H}$,
these adjoint processes $\lambda_{t}$ and $Z_{t}$ do not encode
the standard PMP optimal costate condition. Instead, they represent
the local sensitivities of the resulting objective $J\left(\pi_{t},C_{t}\right)$
with respect to changes in the state $X_{t}$ under the current suboptimal
policy. For this reason, one often refers to them as ``suboptimal''
or ``policy-fixed'' adjoint processes. A comparison between the
optimal policy case and the suboptimal policy case is as follows:
\begin{enumerate}[label=(\alph*)]
\item \textbf{Optimal Policy Case}: If $\left(\pi_{t}^{*},C_{t}^{*}\right)$
actually maximizes $\mathcal{H}$ at each instant, then $\left(\lambda_{t}^{*},Z_{t}^{*}\right)$
coincide with the usual PMP adjoint processes, satisfying $\tfrac{\partial\mathcal{H}}{\partial\pi}=0$,
$\tfrac{\partial\mathcal{H}}{\partial C}=0$, etc. 
\item \textbf{Suboptimal Policy Case}: Here, $\tfrac{\partial\mathcal{H}}{\partial\pi}\neq0$
or $\tfrac{\partial\mathcal{H}}{\partial C}\neq0$, so $\left(\lambda_{t},Z_{t}\right)$
do not necessarily satisfy the typical PMP coupled conditions. Nonetheless,
the BSDE \eqref{eq:BSDE_suboptimal} remains well-defined under standard
Lipschitz and integrability assumptions, thus clarifying how ``the
current policy plus small perturbations in $X_{t}$'' affects the
cost functional. This perspective will be crucial in the gradient-based
updates for suboptimal (parameterized) policies, discussed in the
next sections.
\end{enumerate}

\subsubsection{Iterative Policy Improvement \label{sec:gradient}}

Recall from Section~\ref{sec:Merton} that we parametrize
the investment and consumption controls as neural networks, $\pi_{\theta}(t,X_{t})$
and $C_{\phi}\left(t,X_{t}\right)$, where $\left(\theta,\phi\right)$ collect
all network parameters. By these definitions, the overall performance
$J\left(\theta,\phi\right)$, introduced in \eqref{eq:objective}, depends
on $\left(\theta,\phi\right)$ through the drift and diffusion of
$X_{t}$, as well as through direct utility terms involving $C_{\phi}$.
A central observation is that each small change in $(\theta,\phi)$---through
these neural-network-based controls $\pi_{\theta}(t,X_{t})$ and $C_{\phi}(t,X_{t})$---affects
the drift and diffusion of $X_{t}$ and thus the overall performance
$J\left(\theta,\phi\right)$. To capture this dependence efficiently,
one introduces a pair $(\lambda_{t},Z_{t})$ via a BSDE (often
called the ``policy-fixed adjoint''), where $\lambda_{t}$ acts much
like $\frac{\partial J}{\partial X_{t}}$ and $Z_{t}$ encodes sensitivity
to noise.

As described in \eqref{eq:SDE_state_X}, the state process $X_{t}$
(total wealth) can be rewritten as 
\[
dX_{t}=b\left(X_{t};\theta,\phi\right)dt+\sigma\left(X_{t};\theta,\phi\right)dW_{t},
\]
where 
\begin{equation}
\begin{aligned}b\left(X_{t};\theta,\phi\right) & =rX_{t}+\pi_{\theta}\left(t,X_{t}\right)\,\left(\mu-r\right)X_{t}-C_{\phi}\left(t,X_{t}\right),\\
\sigma\left(X_{t};\theta,\phi\right) & =\sigma\pi_{\theta}\left(t,X_{t}\right)X_{t}.
\end{aligned}
\label{eq:drift_diffusion_X}
\end{equation}

\paragraph{Gradient with respect to $\theta$.}

When we differentiate $b$ and $\sigma$ in \eqref{eq:drift_diffusion_X}
with respect to $\theta$, only the terms involving $\pi_{\theta}$
matter (since $C_{\phi}$ does not depend on $\theta$). Concretely,
\begin{equation}
\begin{aligned}\frac{\partial}{\partial\theta}b\left(X_{t};\theta,\phi\right) & =\left(\mu-r\right)X_{t}\frac{\partial\pi_{\theta}\left(t,X_{t}\right)}{\partial\theta},\\
\frac{\partial}{\partial\theta}\sigma\left(X_{t};\theta,\phi\right) & =\sigma X_{t}\frac{\partial\pi_{\theta}\left(t,X_{t}\right)}{\partial\theta}.
\end{aligned}
\label{eq:dtheta_b_sigma}
\end{equation}
Hence, if $\pi_{\theta}$ is implemented as a neural network, these
partial derivatives reflect the network's internal
chain rule with respect to $\theta$.

Next, recall that $\lambda_{t}=\frac{\partial J}{\partial X_{t}}$
and $Z_{t}$ encodes noise sensitivity. By a standard chain rule argument
in stochastic control, one obtains the general formula 
\begin{equation}
\frac{\partial J}{\partial\theta}=\mathbb{E}\left[\int_{0}^{T}\left(\lambda_{t}\,\frac{\partial b}{\partial\theta}+Z_{t}\,\frac{\partial\sigma}{\partial\theta}\right)\,dt\right]+\text{\ensuremath{\left(\text{direct payoff dependence on \ensuremath{\theta}}\right)}},\label{eq:generic_chainrule}
\end{equation}
where the ``direct payoff'' term collects any explicit dependence
of the utility or terminal cost on $\theta$. However, in many Merton-like
settings, neither $U\left(\cdot\right)$ nor $X_{T}$ depends directly
on $\theta$ (they depend on $\theta$ only through $\pi_{\theta}$),
so that term vanishes. Substituting $\frac{\partial b}{\partial\theta}$
and $\frac{\partial\sigma}{\partial\theta}$ from \eqref{eq:dtheta_b_sigma}
yields 
\[
\begin{aligned}\frac{\partial J}{\partial\theta} & =\mathbb{E}\left[\int_{0}^{T}\left(\lambda_{t}\,\left(\mu-r\right)X_{t}\frac{\partial\pi_{\theta}}{\partial\theta}+Z_{t}\,\sigma X_{t}\frac{\partial\pi_{\theta}}{\partial\theta}\right)\,dt\right]\\
 & =\mathbb{E}\left[\int_{0}^{T}\left(\lambda_{t}\,\left(\mu-r\right)X_{t}+Z_{t}\,\sigma X_{t}\right)\frac{\partial\pi_{\theta}}{\partial\theta}\,dt\right].
\end{aligned}
\]

\paragraph{Gradient with respect to $\phi$.}

First, let us revisit the state-dynamics perspective. From \eqref{eq:drift_diffusion_X},
we have 
\[
\frac{\partial b}{\partial\phi}=-\frac{\partial C_{\phi}\left(t,X_{t}\right)}{\partial\phi},\quad\frac{\partial\sigma}{\partial\phi}=0.
\]
Next, from the payoff perspective, $U\left(C_{\phi}\left(t,X_{t}\right)\right)$
directly depends on $\phi$, so it contributes an extra term $e^{-\rho t}U'\left(C_{\phi}\left(t,X_{t}\right)\right)\,\frac{\partial C_{\phi}}{\partial\phi}$.
Combining these,
\begin{align*}
\frac{\partial J}{\partial\phi} & =\underbrace{\mathbb{E}\left[\int_{0}^{T}\lambda_{t}\Bigl(-\,\tfrac{\partial C_{\phi}}{\partial\phi}\Bigr)\,dt\right]}_{\substack{\text{drift contribution}}
}\;+\;\underbrace{\mathbb{E}\left[\int_{0}^{T}e^{-\rho t}U'\left(C_{\phi}\left(t,X_{t}\right)\right)\,\tfrac{\partial C_{\phi}}{\partial\phi}\,dt\right]}_{\substack{\text{direct payoff contribution}}
}\\
 & =\mathbb{E}\left[\int_{0}^{T}\left(-\,\lambda_{t}+e^{-\rho t}e^{-\rho t}U'\left(C_{\phi}\left(t,X_{t}\right)\right)\right)\,\frac{\partial C_{\phi}(t,X_{t})}{\partial\phi}\,dt\right].
\end{align*}
In words, ${\displaystyle -\lambda_{t}\,\frac{\partial C_{\phi}}{\partial\phi}}$
represents how $C_{\phi}$ reduces the drift in $X_{t}$ (i.e. consumption
lowers wealth), while $e^{-\rho t}U'\left(C_{\phi}\right)\,\frac{\partial C_{\phi}}{\partial\phi}$
captures the direct impact on running utility from changing consumption.

\paragraph{Deriving \eqref{eq:generic_chainrule} via the Chain Rule.}

To see why \eqref{eq:generic_chainrule} holds, recall that $J\left(\theta,\phi\right)$
depends on $\theta$ indirectly through the state process $X_{t}(\theta)$:
\[
J\left(\theta,\phi\right)=\mathbb{E}\left[\int_{0}^{T}g\left(t,X_{t}(\theta),C_{\phi}(t,X_{t})\right)\,dt+G\bigl(X_{T}(\theta)\bigr)\right],
\]
where $g$ and $G$ represent the running cost and the terminal cost,
respectively. By the chain rule, we can conceptually write 
\[
\frac{\partial J}{\partial\theta}=\mathbb{E}\Bigl[\int_{0}^{T}\underbrace{\frac{\partial J}{\partial X_{t}}}_{\text{local sensitivity}}\;\frac{\partial X_{t}}{\partial\theta}\,dt\Bigr]\;+\;\text{(possible direct dependence of \ensuremath{g}, \ensuremath{G} on \ensuremath{\theta})}.
\]
Here, $\tfrac{\partial J}{\partial X_{t}}$ captures how small fluctuations
in $X_{t}$ shift the overall performance, while $\tfrac{\partial X_{t}}{\partial\theta}$
describes how $X_{t}$ itself is altered when $\theta$ changes.

Since we define $\lambda_{t}=\frac{\partial J}{\partial X_{t}}$  and
$Z_{t}$ for noise directions through a BSDE, each infinitesimal
change in $\theta$ influences $X_{t}$ via the drift $b$ and diffusion
$\sigma$, i.e.
\[
\frac{\partial X_{t}}{\partial\theta}\;\longmapsto\;\frac{\partial}{\partial\theta}\,b\left(\cdots\right),\quad\frac{\partial}{\partial\theta}\,\sigma\left(\cdots\right).
\]
Hence, substituting $\lambda_{t}=\tfrac{\partial J}{\partial X_{t}}$
and unrolling the chain rule leads precisely to 
\begin{equation}
\frac{\partial J}{\partial\theta}=\mathbb{E}\left[\int_{0}^{T}\left(\lambda_{t}\frac{\partial b}{\partial\theta}+Z_{t}\frac{\partial\sigma}{\partial\theta}\right)\,dt\right]+\text{\text{\ensuremath{\left(\text{direct payoff dependence on \ensuremath{\theta}}\right)}}},
\end{equation}
where any running or terminal payoff that explicitly depends on $\theta$
(beyond $b$ and $\sigma$) is grouped into the direct payoff term.

In a Merton-like setting, one often finds that $\theta$ enters $J$
only via drift/diffusion, so this direct-payoff term vanishes. In
more general problems, however, $\theta$ might appear explicitly
in the payoff as well, adding another summand to \eqref{eq:generic_chainrule}.

For readers interested in a more rigorous derivation of the chain
rule in stochastic settings (including variation of SDEs and the full
FBSDE theory), we refer to classical treatments such
as \citep{yong2012stochastic,ma1999forward,pardouxpeng1990adapted}. These
works provide a comprehensive account of how one rigorously defines
$\tfrac{\partial X_{t}}{\partial\theta}$ in a stochastic calculus
framework and how the BSDE for $(\lambda_{t},Z_{t})$ captures
the costate processes in PMP.

\section{Comparison with Deep BSDE Methods\label{sec:deep_BSDE_comparison}}

In this section, we compare the adjoint-based BSDE approach from our
Pontryagin framework with the value-based BSDE used in deep
BSDE methods (e.g., \citealp{e2017deep,han2018solving,hure2020deep}).
Both lines of work exploit backward equations with neural network
function approximators, but they differ fundamentally in their targets:
our approach focuses on an adjoint process linked to PMP 
while deep BSDE aims at a value function typically derived from a
PDE viewpoint.

\subsection{Value-Based Methods (Deep BSDE)}

\subsubsection{Hamilton--Jacobi--Bellman Equation and the Value Function}

Consider a controlled diffusion process: 
\begin{equation}
dX_{t}=b\left(t,X_{t},\alpha_{t}\right)\,dt+\sigma\left(t,X_{t},\alpha_{t}\right)\,dW_{t},\quad X_{0}=x_{0},\quad t\in\left[0,T\right],\label{eq:SDE_HJB}
\end{equation}
where $\alpha_{t}$ is an admissible control. In a dynamic programming
framework, the optimal value function $V^{*}(t,x)$ is defined as
\[
V^{*}(t,x)=\sup_{\left\{ \alpha_{s}:s\in[t,T]\right\} }\mathbb{E}\left[\int_{t}^{T}f\left(s,X_{s},\alpha_{s}\right)\,ds+g\bigl(X_{T}\bigr)\,\Big|\,X_{t}=x\right],
\]
where $f$ is a running reward (or negative cost) and $g$ is a terminal
reward. Under suitable smoothness assumptions, $V^{*}$ satisfies
a HJB equation:
\begin{equation}
\begin{aligned}-\,\frac{\partial V^{*}}{\partial t}\left(t,x\right) & -\sup_{\alpha}\,\left\{ \mathcal{L}^{\alpha}\,V^{*}(t,x)+f\bigl(t,x,\alpha\bigr)\right\} =0,\\[4pt]
V^{*}\left(T,x\right) & =g(x),
\end{aligned}
\label{eq:HJB_equation}
\end{equation}
where $\mathcal{L}^{\alpha}$ is the second-order differential operator
associated with the SDE \eqref{eq:SDE_HJB}. Specifically,
\[
\mathcal{L}^{\alpha}V=b\left(t,x,\alpha\right)V'(x)+\tfrac{1}{2}\sigma^{2}\left(t,x,\alpha\right)V''(x).
\]
Once $V^{*}(t,x)$ is known, the optimal control $\alpha^{*}(t)$
is recovered pointwise by maximizing the Hamiltonian in \eqref{eq:HJB_equation},
i.e., 
\[
\alpha^{*}(t,x)=\underset{\alpha}{\arg\max}\,\left\{ \mathcal{L}^{\alpha}\,V^{*}(t,x)+f\bigl(t,x,\alpha\bigr)\right\} .
\]

\subsubsection{Deep BSDE Formulation}

Rather than discretizing \eqref{eq:HJB_equation} directly, deep
BSDE methods re-express the PDE solution through a FBSDE. In many
cases, one may write a BSDE: 
\begin{equation}
\begin{aligned}dY_{t} & ={\displaystyle -f\left(t,X_{t},\alpha_{t},Y_{t},Z_{t}\right)\,dt+Z_{t}\,dW_{t},}\\[4pt]
Y_{T} & =g\left(X_{T}\right),
\end{aligned}
\label{eq:deep_bsde_diff_form}
\end{equation}
where $Y_{t}$ can be interpreted as $V^{\alpha}(t,X_{t})$, the cost-to-go
under the control $\alpha$. If $\alpha_{t}=\alpha_{t}^{*}$ is the
optimal control at every time $t$, then $V^{\alpha}(t,x)$ coincides
with the optimal value function $V^{*}(t,x)$ from \eqref{eq:HJB_equation},
and solving \eqref{eq:deep_bsde_diff_form} recovers the same $Y_{t}$
as the PDE solution $V^{*}(t,X_{t})$. When $\alpha$ is suboptimal,
$\eqref{eq:deep_bsde_diff_form}$ still holds as a representation
of the suboptimal cost, but not of the true HJB value function.

In a deep BSDE approach, one typically parameterizes $(Y_{t},Z_{t})$
via neural networks $(Y^{\theta},Z^{\theta})$ and simulates forward
the state $X_{t}$ while integrating backward the pair $(Y_{t},Z_{t})$.
Concretely, if $\theta$ collects all network parameters, one might
write 
\[
Y_{t_{k}}^{\theta}=Y^{\theta}\left(t_{k},X_{t_{k}}\right),\quad Z_{t_{k}}^{\theta}=Z^{\theta}\left(t_{k},\,X_{t_{k}}\right),
\]
and discretize \eqref{eq:deep_bsde_diff_form} (e.g., via Euler-Maruyama).
To train, one minimizes a loss function measuring consistency between
these neural-network-based processes $\left(Y_{t_{k}}^{\theta},Z_{t_{k}}^{\theta}\right)$
and the BSDE dynamics. For instance, a mean-square error capturing
the difference 
\[
Y_{t_{k}}^{\theta}-\Bigl(Y_{t_{k+1}}^{\theta}+f\bigl(t_{k},X_{t_{k}},\alpha_{t_{k}},Y_{t_{k}}^{\theta},Z_{t_{k}}^{\theta}\bigr)\,\Delta t-Z_{t_{k}}^{\theta}\,\Delta W_{k}\Bigr)
\]
is computed over sampled paths $\left\{ X_{t_{k}}\right\} $, along
with the terminal condition error $\left|Y_{t_{N}}^{\theta}-g\left(X_{t_{N}}\right)\right|^{2}$.
By iterating gradient-based updates on $\theta$, one obtains $Y_{t}^{\theta}$
that approximates $V^{\alpha}(t,X_{t})$ under the chosen control
sequence $\{\alpha_{t}\}$.

If $\alpha_{t}$ is known to be optimal (e.g., from an external iteration, 
or if $\alpha^{*}$ is given), then $Y_{t}^{\theta}\approx V^{*}(t,X_{t})$.
In that scenario, one can extract $\alpha_{t}^{*}$ by substituting
$V^{*}$ into the HJB equation (or Hamiltonian) and choosing $\alpha^{*}$
to maximize $\mathcal{L}^{\alpha}V^{*}+f(t,x,\alpha)$ at each point.
Consequently, \eqref{eq:deep_bsde_diff_form} directly encodes the
value function rather than the control itself, so the policy is recovered
only after one has identified or approximated the correct $\alpha^{*}$.

\subsection{Adjoint-Based Methods (Pontryagin's Principle)}

Our approach bypasses the need to solve the full value function $V(t,x)$
and instead applies PMP directly. In classical PMP theory \citep{pontryagin1962the,fleming2006controlled},
the optimal control $\alpha_{t}^{*}$\footnote{In the Merton problem, this corresponds to $\alpha_{t}^{*}=\left(\pi_{t}^{*},C_{t}^{*}\right)$.}
is obtained by maximizing a local Hamiltonian $\mathcal{H}\bigl(t,X_{t},\alpha,\lambda_{t}^{*},Z_{t}^{*}\bigr)$,
thereby introducing a BSDE for the optimal adjoint $\left(\lambda_{t}^{*},Z_{t}^{*}\right)$.
Concretely, if we denote the terminal cost by $G(X_{T})$, then 
\[
d\lambda_{t}^{*}=-\partial_{x}\mathcal{H}\bigl(t,X_{t}^{*},\alpha_{t}^{*},\lambda_{t}^{*},Z_{t}^{*}\bigr)\,dt+Z_{t}^{*}\,dW_{t},\quad\lambda_{T}^{*}=\partial_{x}G\bigl(X_{T}^{*}\bigr),
\]
with 
\[
\alpha_{t}^{*}=\underset{\alpha}{\arg\max}\,\mathcal{H}\bigl(t,X_{t}^{*},\alpha,\lambda_{t}^{*},Z_{t}^{*}\bigr).
\]
Here, $\lambda_{t}^{*}=\tfrac{\partial J^{*}}{\partial X_{t}}$ measures
how changes in the state $X_{t}$ affect the optimal cost $J^{*}$,
while $Z_{t}^{*}$ represents noise sensitivity. Substituting $\alpha_{t}^{*}$
back into $\mathcal{H}$ at each time $t$ yields the fully optimal
feedback law in the continuous-time limit.

When $\alpha_{t}\neq\alpha_{t}^{*}$, one can define a so-called policy-fixed
(suboptimal) BSDE by inserting the suboptimal control $\alpha_{t}$
into $\mathcal{H}$. This produces 
\[
d\lambda_{t}=-\partial_{x}\mathcal{H}\bigl(t,X_{t},\alpha_{t},\lambda_{t},Z_{t}\bigr)\,dt+Z_{t}\,dW_{t},\quad\lambda_{T}=\partial_{x}G\bigl(X_{T}\bigr),
\]
where $\bigl(\lambda_{t},Z_{t}\bigr)$ now encode how the chosen (suboptimal)
policy $\alpha_{t}$ influences the costate. Although $\lambda_{t}$
here is not the fully optimal costate, it provides a local sensitivity
framework similar to a policy-fixed BSDE in value-based methods.
Iteratively adjusting $\alpha_{t}$ based on such suboptimal adjoint
information can push $\alpha_{t}$ toward $\alpha_{t}^{*}$, as long
as the updates move in a direction that improves the global objective
$J$.

In practice, we parameterize the control as $\alpha_{\theta}(t,X_{t})$
with neural-network parameters $\theta$, and define an overall objective
$J(\theta)$ capturing the expected payoff (or negative cost). By
computing $\nabla_{\theta}J(\theta)$ via BPTT and using the adjoint
$\left(\lambda_{t},Z_{t}\right)$ under the current $\alpha_{\theta}$,
we can perform gradient-based updates that locally maximize $J$.
Repeating these updates refines $\alpha_{\theta}$ until it approaches
to the optimal $\alpha_{t}^{*}$. Correspondingly, the costate $(\lambda_{t},Z_{t})$
approaches to $(\lambda_{t}^{*},Z_{t}^{*})$, maintaining a consistent
PMP-based BSDE perspective throughout.

In summary, deep BSDE methods approximate the value function by solving
an HJB equation or a BSDE for $V(t, x)$, then infer the policy afterwards.
By contrast, Pontryagin-adjoint approaches solve a different
BSDE for the costate $\left(\lambda_{t},Z_{t}\right)$ and directly
update the policy at each iteration. These perspectives involve distinct
computational trade-offs:
\begin{itemize}
\item Deep BSDE can be advantageous if one explicitly needs value function
$V(t,x)$ across the entire domain or must solve high-dimensional
PDEs where direct HJB methods are infeasible.
\item Pontryagin-adjoint can offer a more streamlined route to an optimal
policy, particularly when multiple controls must be determined (e.g.,
consumption and investment in the Merton problem), where extracting
each control from a single value function might prove cumbersome in
a value-based approach.
\end{itemize}
Overall, one's choice between these methods depends
on whether the value function itself is the main quantity of interest,
or whether they prefer a more direct, adjoint-based path to optimizing
the policy.

\section{Gradient-Based Algorithm for Stationary Point Convergence\label{sec:BPTT}}

In the previous sections, we showed how PMP yields adjoint processes
$(\lambda_{t},Z_{t})$ that guide control optimization in continuous
time. Here, we focus on the practical implementation side: specifically,
how to design a neural-network-based scheme that leverages these adjoints
(or their estimates) to learn an approximately optimal policy for
the Merton problem. We embed the PMP conditions into a discrete-time
training procedure, thereby uniting classical control theory with
modern deep learning frameworks.

Our approach relies on BPTT over a discretized version of the stochastic
system. By treating $(\lambda_{t},Z_{t})$ as suboptimal adjoint processes
associated with the current parameterized policy, each gradient update
moves the policy closer to a Pontryagin-aligned solution. Additionally,
we introduce a regularization, which we call an ``alignment penalty'',
to encourage consumption and investment decisions to stay near the
locally optimal Pontryagin controls at each node. This soft enforcement
can reduce variance, accelerate convergence, and stabilize training.

The remainder of this section is organized as follows:
\begin{itemize}
\item Section~\ref{sec:single_path_approach} revisits the
single-path approach for estimating $\lambda_{t}$ and $Z_{t}$ at
each time-state point $(t,x)$. Although each single sample can be
high variance, it remains an unbiased estimator and underpins the
rest of our method. 
\item Section~\ref{sec:discrete_algo} presents a concrete,
discrete-time algorithm for computing $\nabla_{\theta}J$ and $\nabla_{\phi}J$.
We then perform standard gradient steps to update the policy networks,
leveraging the sampled adjoint information. 
\item Section~\ref{sec:normalization} explains how to include
an alignment penalty that softly enforces consistency between the
learned policy and the Pontryagin-derived controls $\left(\pi^{\mathrm{PMP}},C^{\mathrm{PMP}}\right)$
at each node. 
\item Section~\ref{sec:algo_variants:extended_value} extends
our perspective to a global-in-time, extended value function setup
that handles any initial node $(t_{0},x_{0})$, crucial for covering
a broad domain. We then detail two Pontryagin-Guided DPO algorithmic
variants in Section~\ref{sec:algo_variants}. 
\end{itemize}
Overall, we show that by embedding PMP within a BPTT-based stochastic
gradient framework, one can converge rapidly to a near-optimal strategy
in Merton's problem while retaining the flexibility and scalability
of neural networks.

\subsection{Single-Path Approach for Each $\boldsymbol{(t,x)}$\label{sec:single_path_approach}}

A central insight of our method is that a single forward simulation
(or single path) can produce an unbiased estimate of the adjoint
processes $\left(\lambda_{t},Z_{t}\right)$ for a given $(t,x)$.
Specifically, $\lambda_{t}=\partial J/\partial X_{t}$ measures how
changes in the state $X_{t}$ affect the cost $J$. In one-dimensional
problems like Merton's, we also have 
\begin{equation}
Z_{t}=\sigma\pi_{t}X_{t}\bigl(\partial_{x}\lambda_{t}\bigr).\label{eq:Z_t}
\end{equation}

Using just one forward path for each node $(t_{k},X_{k})$ is attractive
for several reasons. First, although each single path is noisy, the
average over many such paths converges to the true adjoint in expectation
\citep{kushner2003stochastic,borkar2008stochastic}. Second, minimal
overhead is required per sample, because we do not store or simulate
large ensembles for the same node; each node simply spawns exactly
one forward simulation. Third, this setup naturally supports online
adaptation: as parameters evolve (and the random seed changes), we
can continuously generate fresh samples that reflect the newly updated
policy.

However, the main drawback of single-path estimation is variance:
if each $\lambda_{k}$ (and $Z_{k}$) is used only once, gradient
estimates can fluctuate significantly, motivating larger batch sizes
or smoothing techniques. In subsequent sections, we discuss strategies
for mitigating this variance, including the alignment penalty (Section~\ref{sec:normalization}),
which helps stabilize updates. Next, we detail how this single-path
approach integrates into our discrete-time algorithm for computing
$\nabla_{\theta}J$ and $\nabla_{\phi}J$ via BPTT, ultimately enabling
a gradient-based update scheme for the policy parameters.

\subsection{Discrete-Time Algorithm for Gradient Computation\label{sec:discrete_algo}}

We now present a concrete procedure, in discrete time, for computing
both the adjoint processes $(\lambda_{t},Z_{t})$ and the parameter
gradients $\left(\nabla_{\theta}J,\nabla_{\phi}J\right)$. This algorithm
relies on BPTT and handles both consumption and investment in the
Merton problem.
\begin{enumerate}[label=(\alph*)]
\item \textbf{Discretize Dynamics and Objective.} Partition the interval
$[0,T]$ into $N$ steps of size $\Delta t=T/N$. Define $t_{k}=k\Delta t$
for $k=0,\ldots,N$, so that $t_{0}=0$ and $t_{N}=T$. We then approximate
the SDE 
\[
dX_{t}=(rX_{t}+\pi_{t}(\mu-r)X_{t}-C_{t})\,dt+\sigma\pi_{t}X_{t}\,dW_{t}
\]
using Euler--Maruyama scheme: 
\[
X_{k+1}=X_{k}+\left(rX_{k}+\pi_{k}(\mu-r)X_{k}-C_{k}\right)\,\Delta t+\sigma\pi_{k}X_{k}\,\Delta W_{k},
\]
where $\pi_{k}=\pi_{\theta}(t_{k},X_{k})$, $C_{k}=C_{\phi}(t_{k},X_{k})$,
and $\Delta W_{k}\sim\mathcal{N}(0,\Delta t)$. The continuous-time
cost 
\begin{equation}
J(\theta,\phi)=\mathbb{E}\left[\int_{0}^{T}e^{-\rho t}\,U\bigl(C_{t}\bigr)\,dt+\kappa e^{-\rho T}\,U\bigl(X_{T}\bigr)\right]\label{eq:objective_NN}
\end{equation}
is discretized similarly as 
\[
J(\theta,\phi)\approx\mathbb{E}\left[\sum_{k=0}^{N-1}e^{-\rho t_{k}}\,U\bigl(C_{k}\bigr)\,\Delta t+\kappa e^{-\rho T}\,U\bigl(X_{N}\bigr)\right].
\]
\item \textbf{Single Forward Path per $(t_{k},X_{k})$.} At each node $\bigl(t_{k},X_{k}\bigr)$,
we run exactly one forward simulation up to the terminal time $T$.
This yields a single-sample payoff, which is noisy but unbiased. Using
backpropagation afterward gives us an estimate of $\lambda_{k}$ and
$Z_{k}$ under the current parameters $\left(\theta,\phi\right)$.
\item \textbf{Compute $\lambda_{k}$ via BPTT.} In a typical deep-learning
framework (e.g., PyTorch), we build a computational graph from $(\theta,\phi)$
through $\{\pi_{k},C_{k}\}$ and $\{X_{k}\}$ to $J(\theta,\phi)$.
A single call to \verb|.backward()| gives $\nabla_{\theta}J$ and
$\nabla_{\phi}J$, as well as partial derivatives of $J$ with respect
to each $X_{k}$. Identifying $\lambda_{k}=\tfrac{\partial J}{\partial X_{k}}$
is consistent with the Pontryagin interpretation of an adjoint process.
\item \textbf{Obtain $\partial_{x}\lambda_{k}$ and Hence $Z_{k}$.} To
compute $Z_{k}$, we need $\partial_{x}\lambda_{k}$. In practice,
one applies another backward pass or higher-order autodiff
on $\lambda_{k}$. Recalling that
\[
Z_{k}=\sigma\pi_{k}X_{k}\bigl(\partial_{x}\lambda_{k}\bigr),
\]
we emphasize this step does not impose optimality, but provides
the extra derivative crucial for forming $\nabla_{\theta}J$ below.
\item \textbf{Update Network Parameters.} We collect the gradients with
respect to $\theta$ and $\phi$ using, for instance: 
\begin{align}
\nabla_{\theta}J & \approx\mathbb{E}\biggl[\sum_{k=0}^{N-1}\Bigl(\lambda_{k}(\mu-r)X_{k}+Z_{k}\sigma X_{k}\Bigr)\,\frac{\partial\pi_{\theta}(t_{k},X_{k})}{\partial\theta}\,\Delta t\biggr],\label{eq:grad_theta_J}\\[6pt]
\nabla_{\phi}J & \approx\mathbb{E}\biggl[\sum_{k=0}^{N-1}\Bigl(-\lambda_{k}+e^{-\rho t_{k}}U'\bigl(C_{k}\bigr)\Bigr)\,\frac{\partial C_{\phi}(t_{k},X_{k})}{\partial\phi}\,\Delta t\biggr].\label{eq:grad_phi_J}
\end{align}
We then update $(\theta,\phi)$ via a stochastic optimizer (e.g.,
Adam or SGD), iterating until the policy converges to a stationary
solution of $J$.
\end{enumerate}
In the formulas \eqref{eq:grad_theta_J} and \eqref{eq:grad_phi_J},
the expectation $\mathbb{E}[\cdots]$ is computed by sampling $M$
trajectories (or mini-batches) and taking an empirical average: 
\[
\mathbb{E}\Bigl[\cdots\Bigr]\approx\frac{1}{M}\sum_{i=1}^{M}\Bigl[\cdots\Bigr]_{i}.
\]

This algorithm effectively performs BPTT on a discretized Merton problem,
using a single-path approach at each node $\bigl(t_{k},X_{k}\bigr)$.
Although variance can be large for any single sample, it remains unbiased
in expectation. Repeated iteration gradually refines a Pontryagin-aligned
policy, as each gradient step exploits the adjoint process $(\lambda_{k},Z_{k})$.

In practice, one could simply define $\pi_{\theta}$ and $C_{\phi}$,
compute $J(\theta,\phi)$ on sampled trajectories, call \verb|.backward()|,
and let autodiff produce $\left(\nabla_{\theta}J,\nabla_{\phi}J\right)$
in a black-box manner, without
explicitly extracting $\lambda_{k}$ or $\partial_{x}\lambda_{k}$.
Nevertheless, interpreting $(\lambda_{k},Z_{k})$ as suboptimal adjoint
processes offers deeper insight and facilitates techniques like the
alignment penalty in the next subsection that explicitly depend on
$\bigl(\lambda_{k},\partial_{x}\lambda_{k}\bigr)$. 

\subsection{Adjoint-Based Regularization for Pontryagin Alignment\label{sec:normalization}}

Previously, we discussed how to obtain online estimates of $\lambda_{k}$
and $\partial_{x}\lambda_{k}$ and thus derive the Pontryagin-optimal
consumption and investment at each time step. In the Merton problem,
for example, one obtains 
\begin{equation}
C^{\mathrm{PMP}}(t_{k},X_{k})=\bigl(e^{\rho t_{k}}\,\lambda_{k}\bigr)^{-\tfrac{1}{\gamma}},\qquad\pi^{\mathrm{PMP}}(t_{k},X_{k})=-\frac{\mu-r}{\sigma^{2}X_{k}}\times\frac{\lambda_{k}}{\partial_{x}\lambda_{k}}\label{eq:C_pi_PMP}
\end{equation}
from \eqref{eq:C_optimal} and \eqref{eq:pi_optimal}, where $\lambda_{k}=\lambda(t_{k},X_{k})$
and $\partial_{x}\lambda_{k}=\partial_{x}\lambda(t_{k},X_{k})$. These
formulas reflect Pontryagin's first-order conditions under CRRA utility.

However, the neural-network-based policies $\pi_{\theta}$ and $C_{\phi}$
may deviate from these ideal controls. A straightforward way
to nudge the networks closer is to add an alignment penalty that measuring
how far the learned policy is from $\bigl(C^{\mathrm{PMP}},\pi^{\mathrm{PMP}}\bigr)$.
Specifically, we define 
\begin{align}
\mathcal{L}_{\text{align}}(\theta,\phi) & =\beta_{C}\sum_{k}\left|C_{\phi}\bigl(t_{k},X_{k}\bigr)-C^{\mathrm{PMP}}\bigl(t_{k},X_{k}\bigr)\right|+\beta_{\pi}\sum_{k}\left|\pi_{\theta}\bigl(t_{k},X_{k}\bigr)-\pi^{\mathrm{PMP}}\bigl(t_{k},X_{k}\bigr)\right|,\label{eq:loss_align}
\end{align}
where $\beta_{C},\beta_{\pi}>0$ control relative penalty on consumption
versus investment mismatch, and the sums run over sampled nodes $\bigl(t_{k},X_{k}\bigr)$.
Each term penalizes how far the policy strays from the Pontryagin-based
target.

Next, we augment the original objective $J(\theta,\phi)$ by
defining 
\[
\widetilde{J}(\theta,\phi)=J(\theta,\phi)-\mathcal{L}_{\text{align}}(\theta,\phi),
\]
whose gradient combines the usual utility-maximization terms with
additional component pushing $\bigl(C_{\phi},\pi_{\theta}\bigr)$
toward $\bigl(C^{\mathrm{PMP}},\pi^{\mathrm{PMP}}\bigr)$. Symbolically,
\[
\nabla_{(\theta,\phi)}\widetilde{J}(\theta,\phi)=\nabla_{(\theta,\phi)}\Bigl(J(\theta,\phi)-\mathcal{L}_{\text{align}}(\theta,\phi)\Bigr)
\]

By choosing $\beta_{C}$ and $\beta_{\pi}$, one can balance how much
to penalize consumption versus investment mismatch. Moderate settings
for both encourage Pontryagin-like
policies without sacrificing the neural network's expressive
power, whereas larger values enforce stricter adherence to the analytical
optimum but can slow or complicate convergence.

We opt for a soft regularization that gently steers the network
toward locally optimal controls at each iteration. This design preserves
flexibility, allowing the policy to adapt dynamically while leveraging
Pontryagin's insights to reduce drifting in high-dimensional parameter
spaces. Such a strategy parallels other research that embeds theoretical
knowledge via a penalty term rather than a hard constraint.
For example, physics-informed neural networks (PINNs) \citep{raissi2019physics}
incorporate PDE residuals into the loss; PDE-constrained optimization
applies adjoint information with soft constraints \citep{gunzburger2002perspectives,hinze2008optimization};
and RL methods may regularize against an expert policy \citep{ross2011reduction}.
As in these examples, our aim is to retain neural-network adaptability
while preventing excessive deviation from known (or partially known)
solutions.

\subsection{Extended Value Function and Algorithmic Variants \label{sec:extended_value_function}}

\subsubsection{Extended Value Function and Discretized Rollouts \label{sec:algo_variants:extended_value}}

In Section~\ref{sec:discrete_algo}, we introduced
a discretization scheme for the Merton problem over a fixed interval
$[0,T]$. Here, we extend that approach to handle any initial time
$t_{0}$ and wealth $x_{0}$ by defining an extended value function
over a broader domain. 

Many continuous-time problems (including Merton) require a policy
$\bigl(\pi_{\theta},C_{\phi}\bigr)$ valid for any initial condition
$(t_{0},x_{0})$. To address this requirement, we define an extended
value function that integrates over random initial nodes $(t_{0},x_{0})\sim\eta(\cdot)$
in the domain $\mathcal{D}\subset[0,T]\times(0,\infty)$, and then
discretize each resulting trajectory. Here, $\eta$ denotes the distribution
of \LyXThinSpace $(t_{0},x_{0})$. Concretely, we set
\[
\widehat{J}(\theta,\phi)=\mathbb{E}_{(t_{0},x_{0})\sim\eta}\left[\mathbb{E}\Bigl(\int_{t_{0}}^{T}e^{-\rho u}U\bigl(C_{\phi}(u,X_{u})\bigr)\,du+\kappa e^{-\rho T}U\bigl(X_{T}\bigr)\Bigr)\Bigr].\right]
\]
Hence, maximizing $\widehat{J}(\theta,\phi)$ yields a policy $\bigl(\pi_{\theta},C_{\phi}\bigr)$
that applies across all sub-intervals $[t_{0},T]$ and initial
wealth $x_{0}$.

To avoid overfitting to a single initial scenario, we draw $(t_{0},x_{0})$
from a suitable distribution $\eta$. By doing so, the policy $\bigl(\pi_{\theta},C_{\phi}\bigr)$
is trained to handle a broad region of the time--wealth domain rather
than just one initial condition. This sampling strategy is standard
in many PDE-based or RL-like continuous-time methods that need a single
control law $(\pi_{\theta},C_{\phi})$ covering all $(t,x)$. 

To implement the integral in $\widehat{J}(\theta,\phi)$, one can
follow the step-by-step procedure outlined in Algorithm~\ref{algo:pgdpo_no_reg},
which details how to discretize the interval, simulate each path,
and apply backpropagation to update $(\theta,\phi).$

\subsubsection{Gradient-Based Algorithmic Variants\label{sec:algo_variants}}

Having introduced the extended value function over random initial
nodes $(t_{0},x_{0}),$we now propose two core Pontryagin-Guided DPO
(PG-DPO) algorithms for the Merton problem. Although both methods
share the same extended-value rollout, they differ in whether they
include a Pontryagin alignment penalty, and thus whether they require
an additional autodiff pass to retrieve $\lambda_{0}$ and $\partial_{x}\lambda_{0}$.

\begin{enumerate}[label=(\alph*)]
\item \textbf{PG-DPO (No Alignment Penalty).}\\
 This is the baseline scheme (Algorithm~\ref{algo:pgdpo_no_reg}),
which simply maximizes $\widehat{J}(\theta,\phi)$ without regularization.
A single BPTT pass suffices to obtain the approximate gradient
$\nabla_{\theta}\widehat{J},\nabla_{\phi}\widehat{J}$, and we update
$(\theta,\phi)$ until convergence. In classical Merton problems with
strictly concave utility, any stationary point in continuous
time is the unique global optimum, although non-convexities in the
neural parameter space may introduce local minima. Empirically, PG-DPO
typically converges near the known Pontryagin solution, but might
exhibit slower or noisier training.
\item \textbf{PG-DPO-Align (With Alignment Penalty).}\\
 This extended scheme (Algorithm~\ref{algo:pgdpo_align})
incorporates the alignment penalty $\mathcal{L}_{\text{align}}$ to
keep $(\pi_{\theta},C_{\phi})$ near Pontryagin-optimal solution $(\pi^{\mathrm{PMP}},C^{\mathrm{PMP}})$
in \eqref{eq:C_pi_PMP}. We thus require a second BPTT (autodiff)
pass to extract $\lambda_{0}=\partial J/\partial X_{0}$ and $\partial_{x}\lambda_{0}$,
from which we compute 
\begin{align*}
C_{0}^{\mathrm{PMP}} & \coloneqq C^{\mathrm{PMP}}(t_{0},X_{0})=\bigl(e^{\rho t_{0}}\,\lambda_{0}\bigr)^{-\tfrac{1}{\gamma}},\\
\pi_{0}^{\mathrm{PMP}} & \coloneqq\pi^{\mathrm{PMP}}(t_{0},X_{0})=-\frac{\mu-r}{\sigma^{2}X_{0}}\times\frac{\lambda_{0}}{\partial_{x}\lambda_{0}}.
\end{align*}
We then form the augmented objective 
\[
\widehat{J}_{\text{align}}(\theta,\phi)=\widehat{J}(\theta,\phi)-\beta_{C}\left|C_{\phi}\bigl(t_{0},X_{0}\bigr)-C_{0}^{\mathrm{PMP}}\right|+\beta_{\pi}\left|\pi_{\theta}\bigl(t_{0},X_{0}\bigr)-\pi_{0}^{\mathrm{PMP}}\right|,
\]
and update $(\theta,\phi)$ to maximize $\widehat{J}_{\text{align}}$.
This alignment step can often yield more stable or faster convergence
(though at higher computational cost).
\end{enumerate}
Note that both PG-DPO and PG-DPO-Align extend naturally to random
initial nodes $(t_{0},x_{0})$, enabling the learned policy to cover
the entire domain of interest. Moreover, we need only $\lambda_{0}$
(rather than the entire $\{\lambda_{k}\}_{k=0}^{N}$) in Algorithm~\ref{algo:pgdpo_align},
because we choose $X_{0}$ (the initial wealth) from a distribution
that covers the relevant domain of $X_{t}$, so performing BPTT solely
with respect to $X_{0}$ suffices. Thus, automatic differentiation
with respect to $X_{0}$ internally accounts for how $\{X_{k}\}_{k=1}^{N}$
evolve, and we can directly obtain 
\[
\lambda_{0}=\frac{\partial J}{\partial X_{0}},\quad\partial_{x}\lambda_{0}=\frac{\partial^{2}J}{\partial X_{0}^{2}},
\]
which allow us to compute $\pi_{0}^{\mathrm{PMP}}$ and $C_{0}^{\mathrm{PMP}}$
for the alignment penalty. Consequently, this saves memory, simplifies
the code, and still provides a rich set of $\lambda_{0}$ estimates
across the domain.

\newpage

\begin{algorithm}[H]
\caption{\textbf{PG-DPO (No Alignment Penalty)}}
\label{algo:pgdpo_no_reg} \textbf{Inputs:} 

\begin{itemize}
\item Policy nets $(\pi_{\theta},\,C_{\phi})$ 
\item Step sizes $\{\alpha_{k}\}$, total iterations $K$ 
\item Domain sampler for $\bigl(t_{0}^{(i)},x_{0}^{(i)}\bigr)$ over $\mathcal{D}\subset[0,T]\times(0,\infty)$ 
\item Fixed integer $N$ (number of time steps per path) 
\end{itemize}
\begin{algorithmic}[1] \FOR{$j=1$ to $K$} \STATE \textbf{(a)
Sample mini-batch of size $M$:} For each $i\in\{1,\dots,M\}$, draw
initial $(t_{0}^{(i)},x_{0}^{(i)})$ from $\mathcal{D}$.

\STATE \textbf{(b) Local Single-Path Simulation for each $i$:} 

\begin{enumerate}[label=(\alph*)]
\item \textit{Define local step:} $\Delta t^{(i)}\,\leftarrow\,\frac{\,T-t_{0}^{(i)}\,}{N},\quad t_{0}^{(i)}<t_{1}^{(i)}<\cdots<t_{N}^{(i)}=T,$
with $t_{k}^{(i)}=t_{0}^{(i)}+k\,\Delta t^{(i)}.$ 
\item \textit{Initialize wealth:} $X_{0}^{(i)}\,\leftarrow\,x_{0}^{(i)}.$ 
\item \textit{Euler--Maruyama:} For $k=0,\dots,N-1$: 
\[
\begin{aligned}X_{k+1}^{(i)}= & X_{k}^{(i)}+\Bigl[rX_{k}^{(i)}+\pi_{\theta}\bigl(t_{k}^{(i)},X_{k}^{(i)}\bigr)(\mu-r)X_{k}^{(i)}-C_{\phi}\bigl(t_{k}^{(i)},X_{k}^{(i)}\bigr)\Bigr]\Delta t^{(i)}\\
 & \qquad\quad+\sigma\pi_{\theta}(t_{k}^{(i)},X_{k}^{(i)})X_{k}^{(i)}\Delta W_{k}^{(i)}.
\end{aligned}
\]
where $\Delta W_{k}^{(i)}\!\sim\!\mathcal{N}\bigl(0,\Delta t^{(i)}\bigr)$. 
\item \textit{Local cost:} 
\[
J^{(i)}(\theta,\phi)=\sum_{k=0}^{N-1}e^{-\rho t_{k}^{(i)}}U\Bigl(C_{\phi}\bigl(t_{k}^{(i)},X_{k}^{(i)}\bigr)\Bigr)\Delta t^{(i)}+\kappa e^{-\rho T}U\bigl(X_{N}^{(i)}\bigr).
\]
\end{enumerate}
\STATE \textbf{(c) Backprop (single pass) \& Averaging}: 
\[
\widehat{J}(\theta,\phi)=\frac{1}{M}\sum_{i=1}^{M}J^{(i)}(\theta,\phi),\qquad\nabla_{(\theta,\phi)}\widehat{J}\;\longleftarrow\;\text{BPTT on each }J^{(i)}.
\]

\STATE \textbf{(d) Parameter Update}: 
\[
\theta\;\leftarrow\;\theta+\alpha_{k}\nabla_{\theta}\widehat{J},\quad\phi\;\leftarrow\;\phi+\alpha_{k}\nabla_{\phi}\widehat{J}.
\]
\ENDFOR \RETURN Final policy $(\pi_{\theta},C_{\phi})$. \end{algorithmic} 
\end{algorithm}

\newpage

\begin{algorithm}[H]
\caption{\textbf{PG-DPO-Align (With Alignment Penalty)}}
\label{algo:pgdpo_align} \textbf{Additional Inputs}: 

\begin{itemize}
\item Alignment weights $(\alpha_{C},\alpha_{\pi})$; 
\item Suboptimal adjoint $(\lambda_{0}^{(i)},\partial_{x}\lambda_{0}^{(i)})$
from a second BPTT pass. 
\end{itemize}
\begin{algorithmic}[1] \FOR{$j=1$ to $K$} \STATE \textbf{(a)
Domain Sampling \& Single-Path Simulation}: As in PG-DPO, collect
each $J^{(i)}(\theta,\phi)$. \STATE \textbf{(b) Retrieve Adjoint
Info (2nd pass)}: 

\begin{enumerate}[label=\roman*)]
\item For each path $i$, run extra autodiff to get $\lambda_{0}^{(i)}=\frac{\partial\,J^{(i)}}{\partial X_{0}^{(i)}}$
and $\partial_{x}\lambda_{0}^{(i)}$. 
\item Pontryagin controls: 
\[
C_{0}^{\mathrm{PMP},(i)}=\bigl(e^{\rho t_{0}}\lambda_{0}^{(i)}\bigr)^{-\frac{1}{\gamma}},\quad\pi_{0}^{\mathrm{PMP},(i)}=-\frac{\mu-r}{\sigma^{2}X_{0}^{(i)}}\frac{\lambda_{0}^{(i)}}{\partial_{x}\lambda_{0}^{(i)}}.
\]
\end{enumerate}
\STATE \textbf{(c) Full Objective with Alignment Penalty}: 
\begin{align*}
\widehat{J}_{\mathrm{align}}(\theta,\phi) & =\frac{1}{M}\sum_{i=1}^{M}\left[J^{(i)}(\theta,\phi)-\beta_{C}\left|C_{\phi}^{(i)}\bigl(t_{0},X_{0}\bigr)-C_{0}^{\mathrm{PMP},(i)}\right|\right.\\
 & \text{\qquad\qquad\qquad\qquad\quad \ensuremath{\left. {}+\beta_{\pi}\left|\pi_{\theta}^{(i)}\bigl(t_{0},X_{0}\bigr)-\pi_{0}^{\mathrm{PMP},(i)}\right|\right]}}
\end{align*}
\STATE \textbf{(d) }:\textbf{ BPTT \& Update}: 
\[
\nabla_{\theta}\widehat{J}_{\mathrm{align}},\,\nabla_{\phi}\widehat{J}_{\mathrm{align}}\;\longleftarrow\;\text{BPTT on each term.}
\]
\[
\theta\leftarrow\theta+\alpha_{0}\nabla_{\theta}\widehat{J}_{\mathrm{align}},\quad\phi\leftarrow\phi+\alpha_{0}\nabla_{\phi}\widehat{J}_{\mathrm{align}}.
\]
\ENDFOR \RETURN Final policy $(\pi_{\theta},C_{\phi})$ balancing
$\widehat{J}$ and alignment penalty. \end{algorithmic} 
\end{algorithm}

\newpage

\section{Stochastic Approximation and Convergence Analysis\label{sec:convergence_analysis}}

In this section, we establish a self-contained convergence result
for our stochastic approximation scheme and then extend it to the
augmented objective that includes the alignment penalty from Section~\ref{sec:normalization}.
Both arguments build on classical Robbins--Monro theory, ensuring
that, under standard assumptions, our gradient-based updates converge
almost surely to a stationary point.

\subsection{Convergence Analysis for the Baseline Objective $\boldsymbol{J(\theta,\phi)}$\label{subsec:6.1_baseline_convergence}}

We begin with the baseline problem of maximizing $J(\theta,\phi)$
in \eqref{eq:objective_NN}, where $(\theta,\phi)$ parametrize the
policy networks $\pi_{\theta}$ and $C_{\phi}$. In practice, as discussed
in Section~\ref{sec:discrete_algo}, the continuous-time
Merton objective is discretized by a time-step $\Delta t=T/N$ and
approximated by a mini-batch or single-path sampling scheme, thus
yielding gradient estimates $\hat{g}_{k}(\theta,\phi)$. Strictly
speaking, when $\Delta t>0$ is finite, a small residual discretization
bias may remain. However, as $\Delta t\,\rightarrow\,0$, this bias
diminishes, and under a well-designed mini-batch sampling, $\hat{g}_{k}$
is regarded as an unbiased (or nearly unbiased) estimator of $\nabla J$.

\paragraph{Key Assumptions (cf. \citealp{kushner2003stochastic,borkar2008stochastic})}
\begin{enumerate}
\item[\textbf{(A1)}] \textbf{(Lipschitz Gradients)} The mapping $\nabla J(\theta,\phi)$
is globally Lipschitz on a compact domain $\Theta\times\Phi$. That
is, there exists some constant $L_{J}>0$ such that 
\[
\|\nabla J(\theta,\phi)-\nabla J(\theta',\phi')\|\le L_{J}\|(\theta,\phi)-(\theta',\phi')\|
\]
for all $(\theta,\phi)$ and $(\theta',\phi')$ in $\Theta\times\Phi$.
\item[\textbf{(A2)}] \textbf{(Unbiased, Bounded-Variance Gradients)} For each iteration
$k$, the gradient estimate $\hat{g}_{k}(\theta,\phi)$ satisfies
\[
\mathbb{E}\bigl[\hat{g}_{k}(\theta,\phi)\bigr]=\nabla J(\theta,\phi)+\delta_{k},\quad\mathbb{E}\bigl[\|\hat{g}_{k}(\theta,\phi)\|^{2}\bigr]\le B,
\]
where $\|\delta_{k}\|\to0$ as $\Delta t\to0$ (eliminating discretization
bias). In typical mini-batch SGD, $\hat{g}_{k}$ is unbiased at each
iteration (assuming i.i.d. sampling from the underlying distribution),
and the variance is uniformly bounded by $B>0$.
\item[\textbf{(A3)}] \textbf{(Robbins--Monro Step Sizes)} The step sizes $\{\alpha_{k}\}$
satisfy 
\[
\sum_{k=0}^{\infty}\alpha_{k}=\infty,\qquad\sum_{k=0}^{\infty}\alpha_{k}^{2}<\infty.
\]
This ensures a standard Robbins--Monro (stochastic gradient) iteration.
\end{enumerate}
Under these assumptions, our parameter update is 
\[
(\theta_{k+1},\phi_{k+1})=(\theta_{k},\phi_{k})+\alpha_{k}\hat{g}_{k}(\theta_{k},\phi_{k}),
\]
where the term $\delta_{k}$ (if any) shrinks as $\Delta t\to0$.

\begin{theorem}[Baseline Robbins--Monro Convergence]\label{thm:stationary}
Suppose \textbf{(A1)}--\textbf{(A3)} hold, and that $\|\delta_{k}\|\to0$
as $\Delta t\to0$. Then, with probability one, 
\[
(\theta_{k},\phi_{k})\;\xrightarrow[k\to\infty]{a.s.}\;(\theta^{\dagger},\phi^{\dagger}),
\]
where $\nabla J(\theta^{\dagger},\phi^{\dagger})=0$. In other words,
$(\theta^{\dagger},\phi^{\dagger})$ is a stationary point of $J$
in the parameter space. \end{theorem}
\begin{proof}[Proof Sketch]
As $\Delta t\to0$, each $\hat{g}_{k}(\theta,\phi)$ becomes (approximately)
unbiased with bounded variance. By the classical Robbins--Monro argument
\citep{kushner2003stochastic,borkar2008stochastic}, the iterates
$(\theta_{k},\phi_{k})$ converge almost surely to a point where $\nabla J=0$.
\end{proof}
This result ensures that our parameter sequence converges to a stationary
point in the parameter space $(\theta,\phi)$. However, because $\pi_{\theta}$
and $C_{\phi}$ are encoded by neural networks, the objective surface
can be non-convex \citep{choromanska2015loss,goodfellow2016deep},
potentially admitting multiple local optima. Meanwhile, in classical
Merton settings with strictly concave utility, the continuous-time
control space $\bigl(\pi_{t},C_{t}\bigr)$ has a unique global
optimum for the original objective. Empirically, we observe that
for typical Merton parameters (e.g., CRRA utility), the learned neural
policy aligns closely with this global solution, indicating that
even though the parametric surface is non-convex, the algorithm tends
to find a near-global optimum in practice.

\subsection{Convergence Analysis for the Augmented Objective $\boldsymbol{\widetilde{J}(\theta,\phi)}$\label{sec:extended_normalization_convergence}}

We now show that introducing the adjoint-based regularization term
(Section~\ref{sec:normalization}) does not invalidate
the convergence guarantees established in Section~\ref{subsec:6.1_baseline_convergence}.
Recall the augmented objective: 
\[
\widetilde{J}(\theta,\phi)=J(\theta,\phi)-\mathcal{L}_{\text{align}}(\theta,\phi),
\]
where $\mathcal{L}_{\text{align}}$ penalizes deviations from Pontryagin-derived
controls $\bigl(\pi^{\mathrm{PMP}},C^{\mathrm{PMP}}\bigr)$. In classical
Merton settings with strictly concave utility, the global optimum
of $J$ also makes $\mathcal{L}_{\text{align}}$ vanish (i.e., no
deviation from the Pontryagin solution), so $\widetilde{J}$ and $J$
share the same unique global maximizer in the continuous-time sense.

Formally, we have 
\begin{equation}
\nabla_{(\theta,\phi)}\widetilde{J}(\theta,\phi)=\nabla_{(\theta,\phi)}J(\theta,\phi)-\nabla_{(\theta,\phi)}\mathcal{L}_{\text{align}}(\theta,\phi).\label{eq:grad_tilde_J}
\end{equation}

\paragraph{Key Assumptions for the Augmented Objective.}

As before, let $\hat{g}_{k}(\theta,\phi)$ be an (approximately) unbiased,
bounded-variance estimator of $\nabla J(\theta,\phi)$, as in (A2).
Now define $\hat{h}_{k}(\theta,\phi)$ analogously as an (approximately)
unbiased, bounded-variance estimator of $\nabla\mathcal{L}_{\text{align}}(\theta,\phi)$.
For convenience, we label the assumptions for $\widetilde{J}$ as
(B1)--(B3), mirroring those of (A1)--(A3):
\begin{enumerate}
\item[\textbf{(B1)}] \textbf{(Lipschitz Gradients for $\widetilde{J}$)} Suppose $\nabla\mathcal{L}_{\text{align}}$
is Lipschitz on the same compact domain $\Theta\times\Phi$. In particular,
if $\pi_{\theta}$ and $C_{\phi}$ each lie in a bounded set with
bounded derivatives, then $\nabla\mathcal{L}_{\text{align}}$ is Lipschitz.
Combined with (A1), it follows that $\nabla\widetilde{J}$ is also
globally Lipschitz.
\item[\textbf{(B2)}] \textbf{(Approximately Unbiased, Bounded-Variance Gradients for $\widetilde{J}$)}
We define 
\[
\hat{f}_{k}(\theta,\phi)=\hat{g}_{k}(\theta,\phi)-\hat{h}_{k}(\theta,\phi),
\]
mirroring \eqref{eq:grad_tilde_J}. As in (A2), each of $\hat{g}_{k}$
and $\hat{h}_{k}$ may exhibit a small bias term $\delta_{k}$ that
vanishes as $\Delta t\to0$. Hence, $\hat{f}_{k}$ is effectively
an unbiased estimator of $\nabla\widetilde{J}$ with uniformly bounded
variance (by some constant $\widetilde{B}>0$) in the limit of fine
discretization.
\item[\textbf{(B3)}] \textbf{(Robbins--Monro Step Sizes)} The step sizes $\{\alpha_{k}\}$
again satisfy 
\[
\sum_{k=0}^{\infty}\alpha_{k}=\infty,\quad\sum_{k=0}^{\infty}\alpha_{k}^{2}<\infty.
\]
This ensures a standard Robbins--Monro iteration for $\widetilde{J}$.
\end{enumerate}
Under these assumptions, the parameter update is 
\[
(\theta_{k+1},\phi_{k+1})=(\theta_{k},\phi_{k})+\alpha_{k}\hat{f}_{k}(\theta_{k},\phi_{k}).
\]

\begin{theorem}[Stationarity of $\widetilde{J}$] \label{thm:stationary_tilde}
Suppose \textbf{(B1)}--\textbf{(B3)} hold. Then, with probability
one, 
\[
(\theta_{k},\phi_{k})\;\xrightarrow[k\to\infty]{a.s.}\;(\theta^{\dagger},\phi^{\dagger}),
\]
where $\nabla\widetilde{J}(\theta^{\dagger},\phi^{\dagger})=0$. In
other words, $(\theta^{\dagger},\phi^{\dagger})$ is a stationary
point of the augmented objective $\widetilde{J}$. \end{theorem}
\begin{proof}[Proof Sketch]
By (B1), $\nabla\widetilde{J}$ is Lipschitz (as it is the sum/difference
of Lipschitz functions). By (B2), $\hat{f}_{k}$ is (approximately)
unbiased with uniformly bounded variance in the limit of $\Delta t\to0$.
By (B3), the step sizes satisfy Robbins--Monro conditions. Hence,
the argument from Theorem~\ref{thm:stationary} applies again, guaranteeing
that $(\theta_{k},\phi_{k})$ converges almost surely to a point where
$\nabla\widetilde{J}=0$.
\end{proof}
Theorem~\ref{thm:stationary_tilde} establishes that even after adding
the alignment penalty $\mathcal{L}_{\text{align}}$, the updated scheme
converges (a.s.) to a stationary point of $\widetilde{J}$. Intuitively,
this reflects a balance between maximizing the original Merton objective
$J$ and staying near the Pontryagin-based controls. In highly non-convex
neural-network models, multiple such stationary points may exist,
so the algorithm may converges to one of them in the parameter
space. However, in classical Merton settings with strictly concave
utility, there is a unique global optimum in continuous time
that makes $\mathcal{L}_{\text{align}}$ vanish. The presence of this
penalty typically helps steer the finite-dimensional parameter updates
more reliably toward that unique global maximizer, thus boosting numerical
stability. Overall, adding $\mathcal{L}_{\text{align}}$ 
neither disrupts the baseline convergence nor alters the fundamental
optimum in strictly concave Merton scenarios.

\section{Numerical Results\label{sec:num_test}}

In this section, we demonstrate our Pontryagin-guided neural approach
on the Merton problem with both consumption and investment,
distinguishing it from prior works that focus primarily on investment
decisions \citep[e.g.,][]{reppen2023deep_df,reppen2023deep_mf,dai2023learning}.
Incorporating consumption notably increases the dimensionality and
difficulty for classical PDE/finite-difference methods.

We consider a one-year horizon $T=1$ and restrict the wealth (state)
domain to $[0.1,\,2]$. The model parameters are $r=0.03$, $\mu=0.12$,
$\sigma=0.2$, $\varepsilon=0.1$, $\gamma=2$, and $\rho=0.02$.
Two neural networks $C_{\phi}$ and $\pi_{\theta}$ approximate the
consumption and portfolio policies, respectively, each taking $(t,X_{t})$
as input so that the entire state-time domain is handled by a single
pipeline without needing to stitch local solutions. Both nets have
two hidden layers (200 nodes each) with Leaky-ReLU activation. We
draw 10,000 initial samples $(t,X_{t})$ uniformly in the time-wealth
domain $\mathcal{D}=[0,1]\times[0.1,2]$. At each iteration, a fresh
batch is simulated in real time via Euler--Maruyama to prevent overfitting
to a single dataset and to ensure the policy remains adapted to the
evolving environmen.

We compare two variants of our scheme:

\begin{enumerate}[label=(\alph*)]
\item \textbf{PG-DPO}: This is the baseline algorithm (Algorithm~\ref{algo:pgdpo_no_reg}),
which uses the discrete-time gradient-based updates from Section~\ref{sec:discrete_algo}
without adding a Pontryagin alignment term. 
\item \textbf{PG-DPO-Align}: This variant adds the penalty $\mathcal{L}_{\text{align}}$
(cf. Section~\ref{sec:normalization}). After brief tuning,
we set $(\beta_{C},\beta_{\pi})=(10^{-3},10^{-1})$), noting that
too large a penalty can slow early training, while too small a penalty
yields minimal improvement. We also allow for different learning rates
for the two policy networks: $1\times10^{-5}$ for consumption, $1\times10^{-3}$
for investment. 
\end{enumerate}
Both methods are trained with the Adam optimizer for up to 100,000
iterations. At intermediate checkpoints (1k, 10k, 50k, 100k), we measure:
(1) the relative mean-squared error (MSE) between the learned
policy and the known closed-form solution, (2) an empirical utility
obtained by averaging the realized payoffs over 500 rollouts centered
around that iteration, to gauge practical performance.

Table~\ref{tab:pgdpo} summarizes the relevant error metrics
and empirical utility outcomes. From this table, we observe that \textbf{PG-DPO-Align}
often converges more stably or quickly, especially at high iteration
counts. 

\begin{table}[]
\centering{\small{}\caption{
\textbf{PG-DPO} vs. \textbf{PG-DPO-Align} in the Merton problem. We compare relative MSEs (consumption/investment) and empirical utility at various iteration milestones. PG-DPO uses step sizes of $1\times10^{-3}$ (investment) and $1\times10^{-5}$ (consumption), while PG-DPO-Align includes an alignment penalty $\mathcal{L}_{\mathrm{align}}$ with $(\beta_{C},\beta_{\pi})=(10^{-3},10^{-1})$ chosen after brief tuning. 
}
\label{tab:pgdpo} }%
\begin{tabular}{cccccc}
\toprule 
\addlinespace[1bp]
\multicolumn{2}{c}{{\small{}Iterations }} & {\small{}1,000} & {\small{}10,000 } & {\small{}50,000} & {\small{}100,000}\tabularnewline\addlinespace[1bp]
\midrule 
\addlinespace[1bp]
\multirow{2}{*}{{\small{}}%
\begin{tabular}{c}
{\small{}Rel. MSE}\tabularnewline
{\small{}(Consumption)}\tabularnewline
\end{tabular}} & {\small{}PG-DPO} & {\small{}3.14e+00} & {\small{}5.73e-01} & {\small{}1.39e-01} & {\small{}9.65e-02}\tabularnewline\addlinespace[1bp]
\cmidrule{2-6} \cmidrule{3-6} \cmidrule{4-6} \cmidrule{5-6} \cmidrule{6-6} 
\addlinespace[1bp]
 & {\small{}PG-DPO-Align} & {\small{}3.00e+00} & {\small{}3.75e-01} & {\small{}6.25e-02} & {\small{}3.46e-02}\tabularnewline\addlinespace[1bp]
\midrule 
\addlinespace[1bp]
\multirow{2}{*}{{\small{}}%
\begin{tabular}{c}
{\small{}Rel. MSE}\tabularnewline
{\small{}(Investment)}\tabularnewline
\end{tabular}{\small{} }} & {\small{}PG-DPO} & {\small{}2.39e-02} & {\small{}2.13e-02} & {\small{}7.85e-03} & {\small{}1.19e-02}\tabularnewline\addlinespace[1bp]
\cmidrule{2-6} \cmidrule{3-6} \cmidrule{4-6} \cmidrule{5-6} \cmidrule{6-6} 
\addlinespace[1bp]
 & {\small{}PG-DPO-Align} & {\small{}7.22e-02} & {\small{}1.57e-02} & {\small{}9.82e-03} & {\small{}8.43e-03}\tabularnewline\addlinespace[1bp]
\midrule 
\addlinespace[1bp]
\multirow{2}{*}{{\small{}}%
\begin{tabular}{c}
Empirical\tabularnewline
Utility\tabularnewline
\end{tabular}} & {\small{}PG-DPO} & {\small{}6.5420e-01} & {\small{}6.4795e-01} & {\small{}6.4791e-01} & {\small{}6.4791e-01}\tabularnewline\addlinespace[1bp]
\cmidrule{2-6} \cmidrule{3-6} \cmidrule{4-6} \cmidrule{5-6} \cmidrule{6-6} 
\addlinespace[1bp]
 & {\small{}PG-DPO-Align} & {\small{}6.5434e-01} & {\small{}6.4793e-01} & {\small{}6.4791e-01} & {\small{}6.4791e-01}\tabularnewline\addlinespace[1bp]
\bottomrule
\end{tabular}
\end{table}

\begin{figure}[ht!]
\centering \begin{subfigure}[b]{0.48\textwidth} \includegraphics[width=1\textwidth]{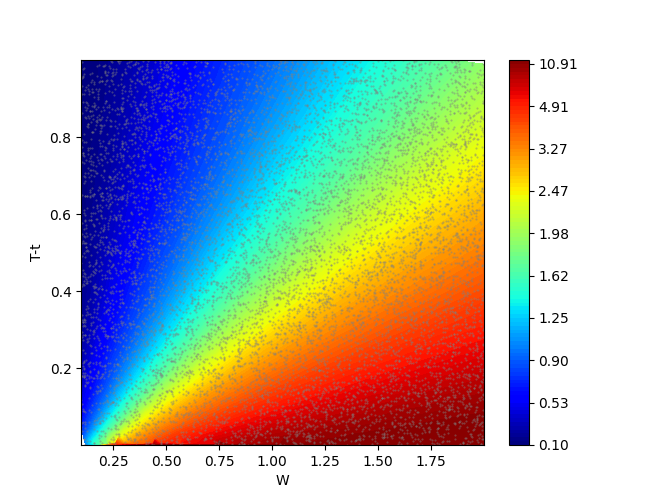}
\caption{Learned consumption policy}
\end{subfigure} \hfill{}\begin{subfigure}[b]{0.48\textwidth}
\includegraphics[width=1\textwidth]{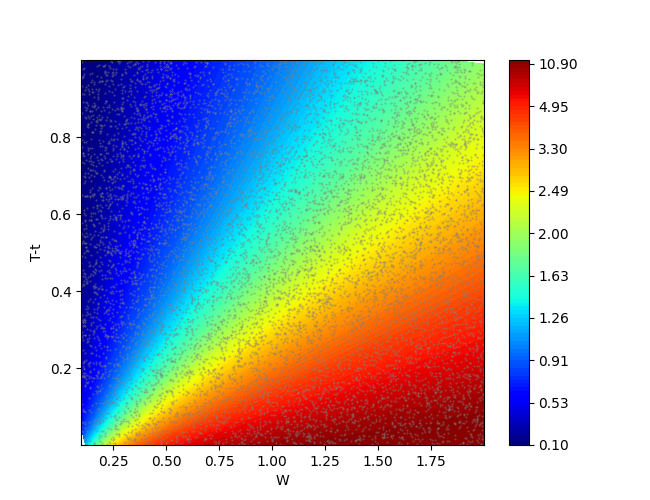} \caption{Exact consumption formula}
\end{subfigure}

\vspace{5pt}

\begin{subfigure}[b]{0.48\textwidth} \includegraphics[width=1\textwidth]{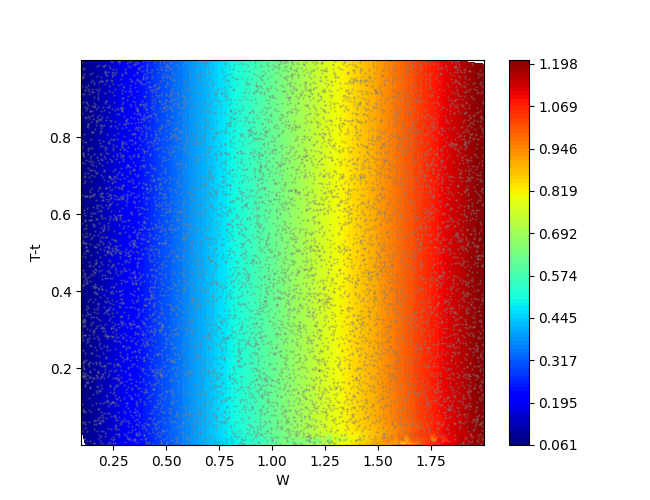}
\caption{Learned investment policy}
\end{subfigure} \hfill{}\begin{subfigure}[b]{0.48\textwidth}
\includegraphics[width=1\textwidth]{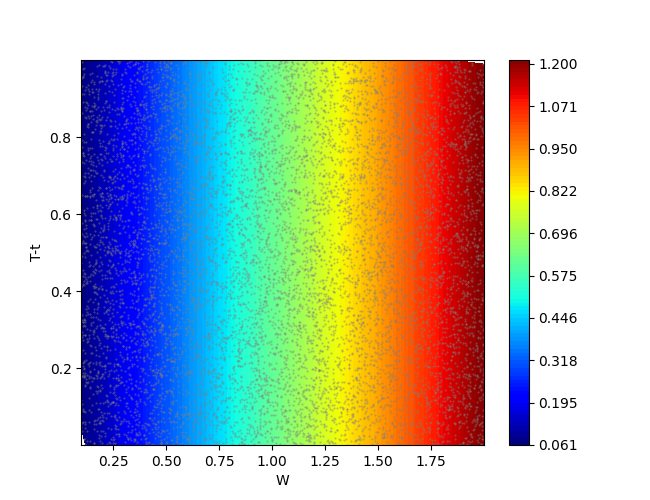} \caption{Exact investment formula}
\end{subfigure}

\caption{{\small{}Neural vs.\ exact solutions for consumption/investment under
the Merton model. The learned plots (left) are from our }\textbf{\small{}PG-DPO-Align}{\small{}
run at iteration 100,000; the exact plots (right) show the
closed-form Merton solution.}}
\label{fig:merton_c_pi} 
\end{figure}

Figure~\ref{fig:merton_c_pi} then compares the learned
policies (after 100k iterations of PG-DPO-Align) to the known closed-form
Merton solution. Despite the additional difficulty from incorporating
consumption, the proposed method converges closely to the theoretical
optimum.

In general, both \textbf{PG-DPO} and \textbf{PG-DPO-Align} eventually reach
high utility levels, with empirical utility converging near the theoretical
maximum. However, \textbf{PG-DPO-Align} yields lower MSE for both consumption
and investment at higher iteration counts (e.g., $3.46\times10^{-2}$ vs. $9.65\times10^{-2}$ 
in consumption MSE at 100k steps), suggesting that the alignment penalty
guides the policy more effectively toward Pontryagin solutions. We also observe
that the alignment variant often stabilizes faster, though the penalty
weights ($\beta_{C}$, $\beta_{\pi}$) must be set moderately to avoid
slowing early progress.

Overall, these results confirm that a Pontryagin-guided alignment
approach can improve both stability and final accuracy in learning
the Merton policy. While the naive gradient method (PG-DPO) ultimately
reaches comparable utility, PG-DPO-Align leverages the local Pontryagin
reference at each step, thereby achieving better performance in practice
and more faithfully matching the known closed-form solution.

\section{Conclusion\label{sec:conclusion}}

We have proposed a PG-DPO framework for Merton's portfolio problem
(Section~\ref{sec:BPTT}), combining neural-network'based
gradient methods with the adjoint perspective from PMP. By parametrizing
both consumption and investment policies within neural networks and
aligning each gradient step to approximate the continuous-time adjoint,
our approach balances theoretical rigor with practical flexibility.

A central design choice is to handle the entire time--wealth domain
by defining an extended value function over random initial nodes $(t_{0},x_{0})$.
This domain-aware sampling allows the policy to learn a truly global
solution rather than overfitting to a single initial state. After
discretizing only the local sub-interval from $t_{0}$ to $T$, we
apply single-path or mini-batch simulation to obtain unbiased gradient
estimates, ensuring that the learned policy covers a broad region
of $[0,T]\times(0,\infty)$.

In addition, we introduce an adjoint-based alignment penalty that
softly regularizes the policy toward Pontryagin's local controls.
Numerical experiments 
confirm that including this penalty can substantially improve stability
and reduce final policy errors, despite the extra cost of a second
autodiff pass to retrieve the suboptimal adjoint processes $\lambda_{t}$
and $\partial_{x}\lambda_{t}$. Our results on the Merton problem
with consumption show that the alignment-penalty variant (PG-DPO-Align)
outperforms the basic scheme (PG-DPO) in terms of convergence speed
and accuracy, while both ultimately reach near-optimal empirical utilities.

Although we do not guarantee global optimality in the high-dimensional
neural parameter space, the strict concavity of the Merton objective
ensures a unique global optimum in continuous time. Empirically, the
extended domain and alignment penalty help guide parameter updates
more reliably toward this Pontryagin-aligned solution, even in the
presence of non-convexities. 

Moreover, to the best of our knowledge, no existing deep-BSDE
or PINN approach addressing Merton's problem with both
consumption and investment in a single framework. By offering a Pontryagin-based
perspective, our PG-DPO method fills this gap and further demonstrates
how continuous-time control principles can enrich modern deep learning
approaches to stochastic optimization.

Further refinements could explore advanced variance-reduction techniques,
expand to higher-dimensional or constrained markets, or combine value-based
PDE elements with the adjoint-based insights. Overall, these findings
illustrate how continuous-time control principles can enrich modern
deep learning approaches to stochastic optimization. By embedding
the local costate viewpoint into a discrete-time training loop, the
PG-DPO framework provides an efficient path to near-optimal controls
for complex continuous-time finance problems and beyond.

\section*{Acknowledgments}

Jeonggyu Huh received financial support from the National Research
Foundation of Korea (Grant No. NRF-2022R1F1A1063371). This work was
supported by the National Research Foundation of Korea (NRF) grant
funded by the Korea government (MSIT) (RS-2024-00355646). 

\bibliographystyle{apalike2}
\bibliography{gf_bib}

\begin{thebibliography}{}

\bibitem[Beck et~al., 2019]{beck2019machine}
Beck, C., E, W., \& Jentzen, A. (2019).
\newblock Machine learning approximation algorithms for high-dimensional fully
  nonlinear partial differential equations and second-order backward stochastic
  differential equations.
\newblock {\em Journal of Nonlinear Science}, 29, 1563--1619.

\bibitem[Becker et~al., 2019]{becker2019deep}
Becker, S., Cheridito, P., \& Jentzen, A. (2019).
\newblock Deep optimal stopping.
\newblock {\em Journal of Machine Learning Research}, 20(74), 1--25.

\bibitem[Borkar \& Borkar, 2008]{borkar2008stochastic}
Borkar, V.~S. \& Borkar, V.~S. (2008).
\newblock {\em Stochastic approximation: a dynamical systems viewpoint},
  volume~9.
\newblock Springer.

\bibitem[Buehler et~al., 2019]{buehler2019deep}
Buehler, H., Gonon, L., Teichmann, J., \& Wood, B. (2019).
\newblock Deep hedging.
\newblock {\em Quantitative Finance}, 19(8), 1271--1291.

\bibitem[Choromanska et~al., 2015]{choromanska2015loss}
Choromanska, A., Henaff, M., Mathieu, M., Arous, G.~B., \& LeCun, Y. (2015).
\newblock The loss surfaces of multilayer networks.
\newblock In {\em Artificial intelligence and statistics}  (pp.\ 192--204).:
  PMLR.

\bibitem[Dai et~al., 2023]{dai2023learning}
Dai, M., Dong, Y., Jia, Y., \& Zhou, X.~Y. (2023).
\newblock Learning merton's strategies in an incomplete market: Recursive
  entropy regularization and biased gaussian exploration.
\newblock {\em arXiv preprint arXiv:2312.11797}.

\bibitem[Fleming \& Soner, 2006]{fleming2006controlled}
Fleming, W.~H. \& Soner, H.~M. (2006).
\newblock {\em Controlled Markov processes and viscosity solutions}, volume~25.
\newblock Springer Science \& Business Media.

\bibitem[Goodfellow, 2016]{goodfellow2016deep}
Goodfellow, I. (2016).
\newblock Deep learning.

\bibitem[Gunzburger, 2002]{gunzburger2002perspectives}
Gunzburger, M.~D. (2002).
\newblock {\em Perspectives in flow control and optimization}.
\newblock SIAM.

\bibitem[Han \& E, 2016]{han2016deep}
Han, J. \& E, W. (2016).
\newblock Deep learning approximation for stochastic control problems.
\newblock {\em arXiv preprint arXiv:1611.07422}.

\bibitem[Han et~al., 2018]{han2018solving}
Han, J., Jentzen, A., \& E, W. (2018).
\newblock Solving high-dimensional partial differential equations using deep
  learning.
\newblock {\em Proceedings of the National Academy of Sciences}, 115(34),
  8505--8510.

\bibitem[Hinze et~al., 2008]{hinze2008optimization}
Hinze, M., Pinnau, R., Ulbrich, M., \& Ulbrich, S. (2008).
\newblock {\em Optimization with PDE constraints}, volume~23.
\newblock Springer Science \& Business Media.

\bibitem[Hur{\'e} et~al., 2020]{hure2020deep}
Hur{\'e}, C., Pham, H., \& Warin, X. (2020).
\newblock Deep backward schemes for high-dimensional nonlinear pdes.
\newblock {\em Mathematics of Computation}, 89(324), 1547--1579.

\bibitem[Karatzas \& Shreve, 1998]{karatzas1998methods}
Karatzas, I. \& Shreve, S.~E. (1998).
\newblock {\em Methods of Mathematical Finance}.
\newblock New York: Springer.

\bibitem[Kushner \& Yin, 2003]{kushner2003stochastic}
Kushner, H.~J. \& Yin, G.~G. (2003).
\newblock {\em Stochastic Approximation and Recursive Algorithms and
  Applications}.
\newblock New York: Springer Science \& Business Media.

\bibitem[Ma \& Yong, 1999]{ma1999forward}
Ma, J. \& Yong, J. (1999).
\newblock {\em Forward-backward stochastic differential equations and their
  applications}.
\newblock Number 1702. Springer Science \& Business Media.

\bibitem[Merton, 1971]{merton1971optimum}
Merton, R.~C. (1971).
\newblock Optimum consumption and portfolio rules in a continuous-time model.
\newblock {\em Journal of Economic Theory}, 3(4), 373--413.

\bibitem[Pardoux \& Peng, 1990]{pardouxpeng1990adapted}
Pardoux, E. \& Peng, S. (1990).
\newblock Adapted solution of a backward stochastic differential equation.
\newblock {\em Systems \& control letters}, 14(1), 55--61.

\bibitem[Pham, 2009]{pham2009continuous}
Pham, H. (2009).
\newblock {\em Continuous-time stochastic control and optimization with
  financial applications}, volume~61.
\newblock Springer Science \& Business Media.

\bibitem[Pontryagin, 2018]{pontryagin1962the}
Pontryagin, L.~S. (2018).
\newblock {\em Mathematical theory of optimal processes}.
\newblock Routledge.

\bibitem[Raissi et~al., 2019]{raissi2019physics}
Raissi, M., Perdikaris, P., \& Karniadakis, G.~E. (2019).
\newblock Physics-informed neural networks: A deep learning framework for
  solving forward and inverse problems involving nonlinear partial differential
  equations.
\newblock {\em Journal of Computational physics}, 378, 686--707.

\bibitem[Reppen \& Soner, 2023]{reppen2023deep_mf}
Reppen, A.~M. \& Soner, H.~M. (2023).
\newblock Deep empirical risk minimization in finance: Looking into the future.
\newblock {\em Mathematical Finance}, 33(1), 116--145.

\bibitem[Reppen et~al., 2023]{reppen2023deep_df}
Reppen, A.~M., Soner, H.~M., \& Tissot-Daguette, V. (2023).
\newblock Deep stochastic optimization in finance.
\newblock {\em Digital Finance}, 5(1), 91--111.

\bibitem[Ross et~al., 2011]{ross2011reduction}
Ross, S., Gordon, G., \& Bagnell, D. (2011).
\newblock A reduction of imitation learning and structured prediction to
  no-regret online learning.
\newblock In {\em Proceedings of the fourteenth international conference on
  artificial intelligence and statistics}  (pp.\ 627--635).: JMLR Workshop and
  Conference Proceedings.

\bibitem[Weinan, 2017]{e2017deep}
Weinan, E. (2017).
\newblock A proposal on machine learning via dynamical systems.
\newblock {\em Communications in Mathematics and Statistics}, 1(5), 1--11.

\bibitem[Yong \& Zhou, 2012]{yong2012stochastic}
Yong, J. \& Zhou, X.~Y. (2012).
\newblock {\em Stochastic controls: Hamiltonian systems and HJB equations},
  volume~43.
\newblock Springer Science \& Business Media.

\bibitem[Zhang \& Zhou, 2019]{zhang2019deep}
Zhang, W. \& Zhou, C. (2019).
\newblock Deep learning algorithm to solve portfolio management with
  proportional transaction cost.
\newblock In {\em 2019 IEEE Conference on Computational Intelligence for
  Financial Engineering \& Economics (CIFEr)}  (pp.\ 1--10).: IEEE.

\end{thebibliography}

\end{document}